\documentclass{amsart}
\usepackage{amssymb,euscript,tikz,units}
\usepackage[colorlinks,citecolor=blue,linkcolor=red]{hyperref}
\usepackage{verbatim}
\usepackage[nameinlink]{cleveref}
\usepackage{colonequals}
\usepackage{tikz}
\usepackage[titletoc]{appendix}
\usepackage{yhmath}
\usepackage{enumerate}
\usepackage{hyperref}
\usepackage{url}

\newcommand{\calT}{\mathcal{T}}
\newcommand{\T}{\calT}
\newcommand{\M}{\mathcal{M}}
\newcommand{\sM}{\mathcal{M}}
\newcommand{\R}{\mathbb{R}}

\newcommand{\Z}{\mathbb{Z}}
\renewcommand{\H}{\mathbb{H}}

\newcommand{\wep}{Weil-Petersson}

\newcommand{\dvol}{d\mathit{vol}}
\newcommand{\E}{\mathbb{E}_{\rm WP}^g}

\newcommand{\ve}{\boldsymbol}

\newcommand{\eg}{\textit{e.g.\@ }}

\def\arcsinh{\mathop{\rm arcsinh}}

\def\Vol{\mathop{\rm Vol}}

\DeclareMathOperator{\Prob}{Prob_{WP}^{g}}

\DeclareMathOperator{\WP}{WP}
\def\Mod{\mathop{\rm Mod}}

\theoremstyle{plain}
\newtheorem{theorem}{Theorem}

\newtheorem{proposition}[theorem]{Proposition}
\newtheorem{lemma}[theorem]{Lemma}

\newtheorem{remark}[theorem]{Remark}
\newtheorem{conjecture}[theorem]{Conjecture}

\newcommand{\be}{\begin{equation}}
\newcommand{\ene}{\end{equation}}
\newcommand{\br}{\begin{remark}}
\newcommand{\er}{\end{remark}}
\newcommand{\bl}{\begin{lem}}
\newcommand{\el}{\end{lem}}
\newcommand{\bcor}{\begin{cor}}
\newcommand{\ecor}{\end{cor}}
\newcommand{\bpro}{\begin{pro}}
\newcommand{\epro}{\end{pro}}
\newcommand{\ben}{\begin{enumerate}}
\newcommand{\een}{\end{enumerate}}
\newcommand{\bp}{\begin{proof}}
\newcommand{\ep}{\end{proof}}
\newcommand{\bpo}{\begin{pro}}
\newcommand{\epo}{\end{pro}}
\newcommand{\beq}{\begin{equation*}}
\newcommand{\eeq}{\end{equation*}}
\newcommand{\bear}{\begin{eqnarray}}
\newcommand{\eear}{\end{eqnarray}}
\newcommand{\beqar}{\begin{eqnarray*}}
\newcommand{\eeqar}{\end{eqnarray*}}
\newcommand{\bt}{\begin{theorem}}
\newcommand{\et}{\end{theorem}}
\newcommand{\bex}{\begin{excer}}
\newcommand{\eex}{\end{excer}}

\theoremstyle{definition}

\theoremstyle{remark}

\newtheorem*{con*}{Construction}
\newtheorem*{rem*}{Remark}
\newtheorem*{exam*}{Example}
\newtheorem*{exams*}{Examples}
\newtheorem*{thm*}{\bf Theorem}
\newtheorem*{que*}{Question}
\newtheorem{que}{Question}
\newtheorem*{Def*}{Definition}
\newtheorem*{Cons*}{Construction}
\newtheorem*{Lem*}{Lemma}
\newtheorem*{Conj*}{\bf Conjecture}
\begin{document}
\title[Arbitrarily small spectral gaps]{Arbitrarily small spectral gaps for random hyperbolic surfaces with many cusps}

\author{Yang Shen}
\address{School of Mathematical Sciences,
Fudan University, Shanghai, China}
\email[(Y.~S.)]{shenwang@fudan.edu.cn}

\author{Yunhui Wu}
\address{Department of Mathematical Sciences and Yau Mathematical Sciences Center, Tsinghua University, Beijing, China}
\email[(Y.~W.)]{yunhui\_wu@tsinghua.edu.cn}

\maketitle

\begin{abstract}
Let $\sM_{g,n(g)}$ be the moduli space of hyperbolic surfaces of genus $g$ with $n(g)$ punctures endowed with the Weil-Petersson metric. In this paper we study the asymptotic behavior of the Cheeger constants and spectral gaps of random hyperbolic surfaces in $\sM_{g,n(g)}$, when $n(g)$ grows slower than $g$ as $g\to \infty$.
\end{abstract}

\section{Introduction}
For a complete non-compact hyperbolic surface $X$ of finite area, the spectrum of the Laplacian operator of $X$ on $L^2(X)$ consists of absolutely continuous spectrum $[\frac{1}{4},\infty)$ and possibly discrete eigenvalues in $[0,\infty)$ (see \eg \cite{LP82}). Although the spectral theory of hyperbolic surfaces has been well studied for decades, the following fundamental question still remains open \cite{Sar03}.
\begin{que}\label{in-q-1}
Does a complete non-compact hyperbolic surface of finite area always have a non-zero eigenvalue?
\end{que}
\noindent Selberg's trace formula \cite{Sel65} shows that certain arithmetic non-compact hyperbolic surfaces have infinitely many eigenvalues in $[\frac{1}{4},\infty)$. But Phillips and Sarnak \cite{PS85} showed that certain eigenvalues disappear when some arithmetic hyperbolic surface is deformed to an ordinary hyperbolic surface. And the \emph{Phillips-Sarnak Conjecture} (\eg \cite{DIPS85, Sar03}) asserts that a generic non-compact hyperbolic surface of finite area except once-punctured tori has only finite many discrete eigenvalues. One may see \cite{Col83,HJu18,Jud95,Luo01,PS92,Wol94} and references therein for related works.

Let $N\geq 1$ and $\Gamma(N)$ be the principal congruence subgroup of $\mathrm{SL(2,\mathbb{Z})}$ of level $N$ defined as $\Gamma(N)\overset{\text{def}}{=}\left\{A \in \mathrm{SL(2,\mathbb{Z})}; \ A \equiv 1 \ \mathrm{mod} \ N \right\}$. Set $X(N)=\H/ \Gamma(N)$. It is known that $X(N)$ is a hyperbolic surface of genus $g(N)$ with $n(N)$ punctures where as $N\to \infty$, $g(N)$ and $n(N)$ roughly grow like $N^3$ and $N^2$ respectively (see \eg \cite[Theorem 2.12]{Ber-book}). So $n=n(g)\asymp g^{\frac{2}{3}}$ for $X(N)$. Let $\lambda_1(X(N))$ be the first non-zero eigenvalue of $X(N)$. A celebrated theorem of Selberg \cite{Sel65} says that $\lambda_1(X(N))\geq \frac{3}{16}$ for all $N\geq 1$. And \emph{Selberg's $\frac{1}{4}$ Conjecture} asserts that the lower bound $\frac{3}{16}$ can be improved to be $\frac{1}{4}$. The best known lower bound is $\frac{975}{4096}$, due to Kim-Sarnak \cite{KS03}. One may also see \eg \cite{GJ78, Iwa89, Iwa96,KS02, LRS95, Sar95} for intermediate results.

Mysteriously, the number $\frac{3}{16}$ also appears as lower bounds for uniform spectral gaps of random hyperbolic surfaces in different types of models. Let $\sM_{g,n}$ be the moduli space of hyperbolic surfaces of genus $g$ with $n$ punctures endowed with the \wep \ metric. Joint with Xue, the second named author \cite{WX22} showed that for any $\epsilon>0$, as $g\to \infty$, a generic surface $X\in \sM_{g,0}$ has first eigenvalue greater than $\frac{3}{16}-\epsilon$, improving Mirzakhani's prior lower bound $0.0024$ \cite{Mirz13}. One may also see Lipnowski-Wright \cite{LW21} for an independent proof. Very recently it has been improved to be $\frac{2}{9}-\epsilon$ by Anantharaman-Monk in \cite{AM23}. For random cusped surfaces, Hide \cite{Hide21} showed that if $n=O(g^\alpha)$ for some $0\leq\alpha<\frac{1}{2}$, as $g\to \infty$, a generic surface $X\in \sM_{g,n}$ has a spectral gap greater than $C(\alpha)-\epsilon$ for any $\epsilon>0$, where $C(\alpha)$ is a positive constant depending on $\alpha$ such that
$$C(0)=\frac{3}{16}\text{\ and }\lim\limits_{\alpha\to \frac{1}{2}}C(\alpha)=0.$$
For covering model, Magee-Naud-Puder \cite{MNP20} showed that as the covering degree tends to infinity, a generic covering surface of a fixed closed hyperbolic surface has relative spectral gap of size $\frac{3}{16}-\epsilon$, which was improved to $\frac{1}{4}-\epsilon$ by Hide-Magee \cite{HM21} for random covers of cusped hyperbolic surfaces. As an important application, Hide and Magee firstly proved the existence of a sequence of closed hyperbolic surfaces with genus going to infinity and first eigenvalues tending to $\frac{1}{4}$. Analogous to Selberg's $\frac{1}{4}$ Conjecture, it is \emph{conjectured} that the lower bounds $\frac{3}{16}-\epsilon$ in \cite{MNP20, WX22, LW21} can be replaced by $\frac{1}{4}-\epsilon$.

It is quite natural to ask (see \eg \cite[Question 1.7]{Hide21} for case that $n \sim g^{\frac{2}{3}}$):
\begin{que}\label{in-q-2}
If $n(g)$ grows significantly faster than $\sqrt g$, asymptotically does a generic surface in $\sM_{g,n(g)}$ have a uniform positive spectral gap as $g\to \infty$?
\end{que}

For a non-compact hyperbolic surface $X$ of finite area, we write $\lambda_1(X)$ as the first non-zero eigenvalue of the Laplacian operator of $X$ \emph{if it exists}. Let $\Prob$ be the probability measure on moduli space $\sM_{g,n(g)}$ given by the Weil-Petersson metric. We prove

\begin{theorem}\label{mt-1}
If $n(g)$ satisfies that $$\lim\limits_{g\to\infty}\frac{n(g)}{\sqrt g}=\infty\text{ and }\lim\limits_{g\to\infty}\frac{n(g)}{g}=0,$$ then for any $\epsilon>0$
$$\lim\limits_{g\to\infty} \Prob\left(X\in\mathcal{M}_{g,n(g)};\ \lambda_1(X)<\epsilon\right)=1.$$
\end{theorem}
\noindent Theorem \ref{mt-1} partially answers Question \ref{in-q-2}. As a special case, it gives a complete answer to \cite[Question 1.7]{Hide21} by Hide. It also tells that for large $g$, the answer to Question \ref{in-q-1} is affirmative for asymptotically generic surfaces in $\sM_{g,n(g)}$ provided that $g^{\frac{1}{2}+\delta}\prec n(g)\prec g^{1-\delta}$ for some $\delta>0$; moreover, the first eigenvalues could be arbitrarily small. For congruence covering surfaces $\{X(N)\}$, Theorem \ref{mt-1} tells that they are exceptional for first eigenvalues in the moduli space of Riemann surfaces homeomorphic to $X(N)$.

\begin{rem*}
It would be \emph{interesting} to know whether the condition that $$\lim\limits_{g\to\infty}\frac{n(g)}{g}=0$$ can be removed in Theorem \ref{mt-1}. We expect a \emph{positive} answer, which will completely solve Question \ref{in-q-2}.  For the case that $\lim\limits_{g\to\infty}\frac{n(g)}{g}=\infty$, by the work \cite[Theorem 2]{Zo84} of Zograf it is known that for $X\in \sM_{g,n(g)}$, $\lambda_1(X) \prec \frac{g}{n(g)}$ which in particular implies that
\[\lim\limits_{g\to \infty}\sup\limits_{X\in \sM_{g,n(g)}}\lambda_1(X)=0.\]
Hence it suffices to consider the case that $g$ and $n(g)$ have the same growth rates.
 We also remark here that Hide-Thomas \cite{HT22} showed that as $n\to \infty$, if  $i(n)=o(n^{\frac{1}{3}})$, the $i(n)$-th eigenvalue of a generic surface in $\sM_{0,n}$ is arbitrarily small.
\end{rem*}

\noindent Let $V_{g,n}$ be the Weil-Petersson volume of $\M_{g,n}$. In the proof of Theorem \ref{mt-1}, the geometric quantity $\frac{V_{g,n}^2}{V_{g,n-1}V_{g,n+1}}$ shows up (see Proposition \ref{ph-1}). And the following question on this quantity is naturally raised:
\begin{que}\label{que-wpv}
Does the following limit hold:
$$\lim\limits_{g+n\to\infty}\frac{V_{g,n}^2}{V_{g,n-1}V_{g,n+1}}=1$$ where $n$ may depend on $g$.
\end{que}

\begin{rem*}
\ben
\item By \cite[Lemma 5.1]{MZ15} of Mirzakhani-Zograf, the answer to Question \ref{que-wpv} is affirmative if $n(g)=o(g)$. One may also see Lemma \ref{l-cor-3}. 

\item If $g$ is fixed, it is known by \cite[Theorem 6.1]{MZ00} of Manin-Zograf that the answer to Question \ref{que-wpv} is also positive. 

\item A positive answer to Question \ref{que-wpv} would imply that the condition that $\lim\limits_{g\to\infty}\frac{n(g)}{g}=0$  in Theorem \ref{mt-1} can be removed. See Sections \ref{sec-proof} and \ref{sec-ques} for details.

\item In Section \ref{sec-ques} we will show in Proposition \ref{p-est} that $\limsup\limits_{g+n\to\infty}\frac{V_{g,n}^2}{V_{g,n-1}V_{g,n+1}}\leq 1$.
\een
\end{rem*}

 As introduced above, the case that $n(g)=o(\sqrt{g})$ was studied by Hide in \cite{Hide21}. Let $h(X)$ be the Cheeger constant of $X\in \sM_{g,n(g)}$. It is known by Cheeger-Buser inequality that $h$ and $\lambda_1$ can bound each other. In this paper, we also study the critical case that $n(g)\asymp \sqrt{g}$, and show that
\begin{theorem}\label{mt-3}
Assume $n(g)$ satisfies that $$\lim\limits_{g\to\infty}\frac{n(g)}{\sqrt g}=a\in (0,\infty).$$ Then for any $0<C<\frac{\log 2}{\sqrt{4\pi(\log 2+\pi)}}$,
$$\lim\limits_{g\to\infty}\Prob\left(X\in\mathcal{M}_{g,n(g)}; \ h(X)\leq \frac{C}{\sqrt{1+C^2}}\right)=1-e^{-\lambda(a,C)}$$
where $$\lambda(a,C)=\frac{a^2}{4\pi^2}\left(\cosh\pi C-1\right).$$
\end{theorem}
\begin{rem*}
\begin{enumerate}
\item For any fixed $\epsilon>0$, it is clear that $$\lim\limits_{a\to\infty}\left(1-e^{-\lambda(a,\epsilon)}\right)=1.$$ 
Then Theorem \ref{mt-3} implies that for any $\epsilon>0$
\[\lim \limits_{a\to \infty}\lim\limits_{g\to \infty}\Prob\left(X\in\mathcal{M}_{g,n(g)}; \ h(X)< \epsilon\right)=1.\]
This together with Buser's inequality \cite[Theorem 7.1]{Bu82} tells that Theorem \ref{mt-1} and  \ref{mt-3} coincide with each other at their concatenation.
\item For any $0<C<\frac{\log 2}{\sqrt{4\pi(\log 2+\pi)}}$, we have $$\lim\limits_{a\to 0}\left(1-e^{-\lambda(a,C)}\right)=0.$$
Then Theorem \ref{mt-3} implies that
\[\lim\limits_{a\to 0^{+}} \lim\limits_{g\to\infty}\Prob\left(X\in\mathcal{M}_{g,n(g)}; \ h(X)> \frac{C}{\sqrt{1+C^2}}\right)=1.\]
This together with the classical Cheeger's inequality \cite[Page 228]{Bu82} tells that Theorem \ref{mt-3} coincides with the following Theorem \ref{mt-4} on the case $n(g)=o(\sqrt{g})$  at their concatenation.
\end{enumerate}
\end{rem*}
\begin{theorem}\label{mt-4}
Assume $n(g)$ satisfies that $$\lim\limits_{g\to\infty}\frac{n(g)}{\sqrt g}=0.$$
Then for any $\epsilon>0$,
$$\lim\limits_{g\to\infty}\Prob\left(X\in\mathcal{M}_{g,n(g)};\ \textnormal{ spec}\left(X\right)\cap \left(0,\frac{1}{4}\left(\frac{\log 2}{\log 2+2\pi}\right)^2-\epsilon\right)=\emptyset\right)=1$$
where $\textnormal{spec}(X)$ is the spectrum of $X$.
\end{theorem}

\noindent This theorem above improves the result of Hide \cite{Hide21} for the case that $n=O(g^\alpha)$ when \emph{$\alpha$ is close to $\frac{1}{2}$}; but it is weaker when $\alpha$ is close to $0$. The proof of Theorem \ref{mt-4} is based on Cheeger's inequality and closely follows the work \cite{Mirz13} of Mirzakhani. However, the Cheeger's inequality is not sharp for hyperbolic surfaces of high genus: one may see recent work \cite{BCP22} of Budzinski-Curien-Petri on Cheeger constants of hyperbolic surfaces for large genus. Inspired by \cite{WX22, LW21, AM23} for closed surface case we conjecture that
\begin{conjecture}
Theorem \ref{mt-4} still holds with replacing $\frac{1}{4}\left(\frac{\log 2}{\log 2+2\pi}\right)^2$ by $\frac{1}{4}$.
\end{conjecture}

Define a function $$y(x)\overset{\text{def}}{=}\liminf\limits_{g\to\infty}\textnormal{Prob}_{\textnormal{WP}}^{g}\left(X\in\mathcal{M}_{g,n(g)};\ \lambda_1(X)\leq x\right).$$
Then we may summarize Theorem \ref{mt-1}, Theorem \ref{mt-3} and Theorem \ref{mt-4} as the following picture:
\begin{figure}[ht]
\centering
\includegraphics[width=0.90\textwidth]{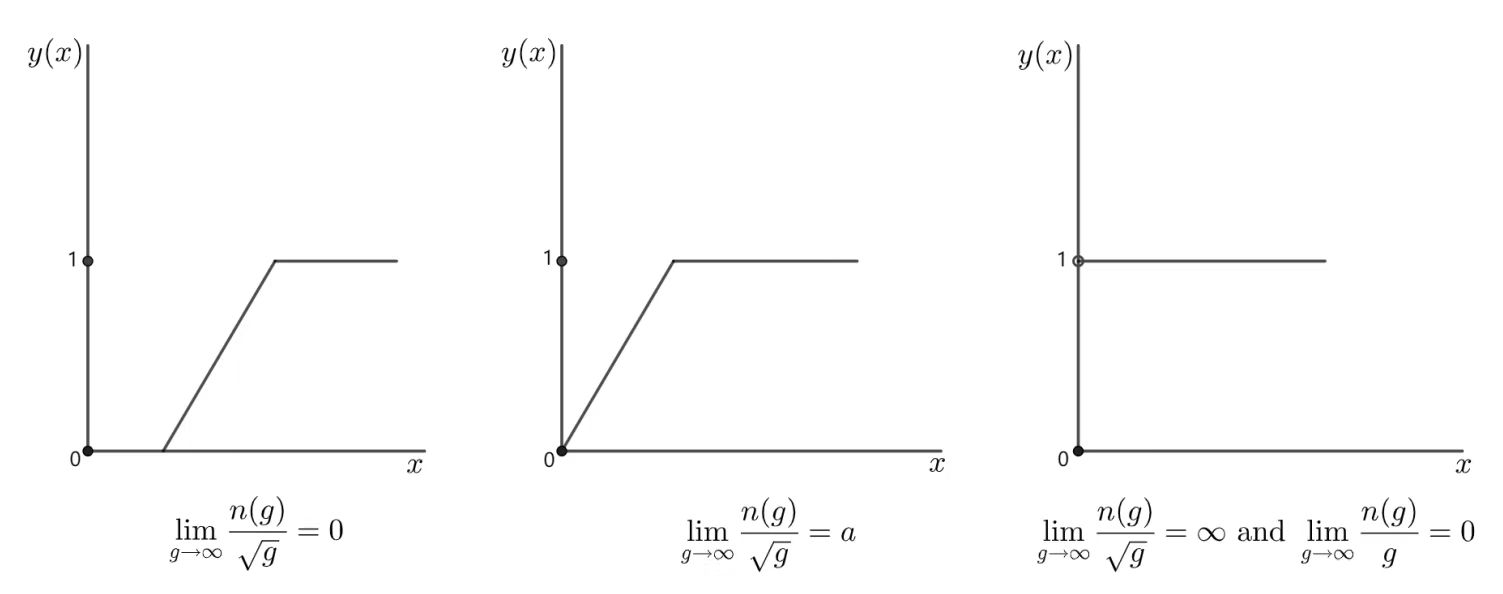}
\caption{Spectral gaps}
\label{fig:03}
\end{figure}
\subsection*{Notations.} For any two non-negative functions $f$ and $h$ (may be of multi-variables), we say $h\succ f$ or $f\prec h$ if there exists a uniform constant $C>0$ such that $f\leq Ch$. And we say $f\asymp h$ if $f\succ h$ and $f\prec h$. For $0\leq m\leq n \in \Z$, we use $\binom{n}{m}$ for binomial coefficient. 

\subsection*{Plan of the paper.} Section \ref{sec-pre} will provide a review of relevant background materials, and recall some useful results for Weil-Petersson measure, spectral theory, Cheeger constant and probability theory. In Section \ref{sec-wp} we will firstly recall some results on Weil-Petersson volume of $\sM_{g,n}$, and then prove some estimations which will be applied later. In Section \ref{multi-curve} we will introduce two special types of separating simple closed multi-curves and estimate the expected values of their related random variables. In Section \ref{sec-proof} we will discuss three different types of $n(g)$ and finish the proofs of Theorem \ref{mt-1}, \ref{mt-3} and \ref{mt-4}. In Section \ref{sec-ques} we will study Question \ref{que-wpv} and give a partial answer to it. 

\subsection*{Acknowledgement}
We are grateful to the ``Reading seminar on Teichm\"uller theory and related topics" at Tsinghua from which this work benefits a lot. We especially would like to thank Yuhao Xue for many helpful discussions on this project, and his comments on an earlier version, and thank Peter Zograf for his useful comments and informing us his early work \cite{Zo84} telling that for large enough $g$, the set with $\lambda_1(X)<\epsilon$ in Theorem \ref{mt-1} is a full set provided that $\lim \limits_{g\to \infty}\frac{n(g)}{g}=\infty$. We are also grateful to Zeev Rudnick for his interests, comments and suggestions on this work. The second named author is partially supported by NSFC grants No. 12171263, 12361141813, and 12425107.


\tableofcontents

\section{Preliminaries}\label{sec-pre}
In this paper we use the same notations as in \cite{NWX20,WX22}. In this section, we review the relevant background materials about Weil-Petersson metric, Mirzakhani's Integration Formula, basic spectral theory for non-compact type hyperbolic surfaces, two elementary probabilistic results, Cheeger constant and so on.


\subsection{Weil-Petersson symplectic form}
Recall that for a surface $S_{g,n}$ of genus $g$ with $n$ punctures where $2g+n\geq 3$, associated to a pants decomposition of $S_{g,n}$, the \emph{Fenchel-Nielsen coordinate}, given by $X\mapsto(\ell_{\alpha_i}(X),\tau_{\alpha_i}(X))_{i=1}^{3g+n-3},$ is a global coordinate for the Teichm\"uller space $\mathcal{T}_{g,n}$ of $S_{g,n}$. Where $\{\alpha_i\}_{i=1}^{3g+n-3}$ are pairwisely disjoint simple closed curves, $\ell_{\alpha_i}(X)$ is the length of $\alpha_i$ on $X$ and $\tau_{\alpha_i}(X)$ is the twist along $\alpha_i$ (measured in length). Wolpert \cite{Wolpert82} showed that the Weil-Petersson symplectic structure has a magic form in Fenchel-Nielsen coordinates. More precisely,
\begin{theorem}[Wolpert]\label{wol-wp}
The Weil-Petersson symplectic form $\omega_{\mathrm{WP}}$ on $\mathcal{T}_{g,n}$ is given by $$\omega_{\textnormal{WP}}=\sum_{i=1}^{3g+n-3}d\ell_{\alpha_i}\wedge d\tau_{\alpha_i}.$$
\end{theorem}
We mainly work with the \emph{Weil-Petersson volume form}
$$
\dvol_{\mathrm{WP}}\overset{\text{def}}{=}\tfrac{1}{(3g+n-3)!}\underbrace{\omega_{\WP}\wedge\cdots\wedge\omega_{\WP}}_{\text{$3g+n-3$ copies}}~.
$$
It is a mapping class group invariant measure on $\T_{g,n}$, hence is the lift of a measure on the moduli space $\M_{g,n}$ of Riemann surfaces of genus $g$ with $n$ punctures, which we also denote by $\dvol_{\WP}$. The total volume of $\M_{g,n}$ is finite and we denote it by $V_{g,n}$. The Weil-Petersson volume form is also well-defined on the weighted moduli space $\M_{g,n}(\ve{L})$ and its total volume, denoted by $V_{g,n}(\ve{L})$, is finite.

Following Mirzakhani \cite{Mirz13}, we view a measurable function $f:\mathcal{M}_{g,n}\to\mathbb{R}$ as a random variable on $\mathcal{M}_{g,n}$ with respect to the probability measure $\textnormal{Prob}_{\textnormal{WP}}^{g,n}$ defined by normalizing $dvol_{\textnormal{WP}}$, and let $\mathbb{E}_{\textnormal{WP}}^{g,n}[f]$ denote the expected value. Namely
$$\textnormal{Prob}_{\textnormal{WP}}^{g,n}(\mathcal{A})\overset{\text{def}}{=}\frac{1}{V_{g,n}}\int_{\mathcal{M}_{g,n}}\textbf{1}_{\mathcal{A}}dX
,\ \mathbb{E}_{\textnormal{WP}}^{g,n}[f]\overset{\text{def}}{=}\frac{1}{V_{g,n}}\int_{\mathcal{M}_{g,n}}f(X)dX,$$
where $\mathcal{A}\subset\mathcal{M}_{g,n}$ is any Borel subset, $\mathbf{1}_{\mathcal{A}}:\mathcal{M}_{g,n}\to\{0,1\}$ is its characteristic function, and we always denote $dvol_{\mathrm{WP}}(X)$ as $dX$ for short. When $n=n(g)$ depends on $g$, we write $\textnormal{Prob}_{\textnormal{WP}}^{g,n}$ and $\mathbb{E}_{\textnormal{WP}}^{g,n}$ as $\Prob$ and $\E$ for short respectively.


\subsection{Mirzakhani's Integration Formula}

In this subsection we recall an integration formula due to Mirzakhani in \cite{Mirz07a, Mirz13}, which plays a role in the study of random surfaces in the \wep \ model.

Given any non-peripheral closed curve $\gamma$ on a topological surface $S_{g,n}$ and a hyperbolic surface $X\in\T_{g,n}$, we denote by $\ell_\gamma(X)$ the hyperbolic length of the unique closed geodesic representing $\gamma$ on $X$. Let $\Gamma=(\gamma_1,\cdots,\gamma_k)$ be an \emph{ordered} $k$-tuple where the $\gamma_i$'s are distinct disjoint homotopy classes of nontrivial, non-peripheral, simple closed curves on $S_{g,n}$. Let $\Mod_{g,n}$ be the mapping class group of $S_{g,n}$. We consider the orbit containing $\Gamma$ under $\Mod_{g,n}$ action
\begin{equation*}
\mathcal O_{\Gamma} = \{(h\cdot\gamma_1,\cdots,h\cdot\gamma_k) ; \ h\in\Mod\nolimits_{g,n}\}.
\end{equation*}
Given a function $F:\R^k_{\geq0} \rightarrow \R_{\geq0}$, one may define a function on $\M_{g,n}$ as follows.
\begin{eqnarray*}
F^\Gamma:\M_{g,n} &\rightarrow& \R \\
X &\mapsto& \sum_{(\alpha_1,\cdots,\alpha_k)\in \mathcal O_\Gamma} F(\ell_{\alpha_1}(X),\cdots,\ell_{\alpha_k}(X)).
\end{eqnarray*}

\noindent Assume $S_{g,n}\setminus\bigcup\limits_{j=1}^k\gamma_j \cong \bigcup\limits_{i=1}^s S_{g_i,n_i}$. We consider the volume
\begin{equation*}
V_{g,n}(\Gamma,\boldsymbol{x}) = \prod_{i=1}^s V_{g_i,n_i}(\boldsymbol{x}^{(i)})
\end{equation*}
where $\boldsymbol{x}^{(i)}$ is the list of those coordinates $x_j$ of $\boldsymbol{x}$ such that $\gamma_j$ is a boundary component of $S_{g_i,n_i}$, and $V_{g_i,n_i}(\boldsymbol{x}^{(i)})$ is the Weil-Petersson volume of the moduli space $\M_{g_i,n_i}(\boldsymbol{x}^{(i)})$. Mirzakhani established the following formula by applying Wolpert's magic formula Theorem \ref{wol-wp}. One may also refer to \cite[Theorem 7.1]{Mirz07a} or \cite[Theorem 2.2]{MP19} for more details on this formula.

\begin{theorem}[Mirzakhani]\label{Mirz-for}
For any $\Gamma=(\gamma_1,\cdots,\gamma_k)$, the integral of $F^\Gamma$ over $\M_{g,n}$ with respect to Weil-Petersson metric is given by
\begin{equation*}
\int_{\M_{g,n}} F^\Gamma(X)dX =
C_\Gamma\int_{\R^k_{\geq0}} F(x_1,\cdots,x_k)V_{g,n}(\Gamma,\boldsymbol{x}) \boldsymbol{x}\cdot d\boldsymbol{x}
\end{equation*}
where $\boldsymbol{x}\cdot d\boldsymbol{x} = x_1\cdots x_k dx_1\wedge\cdots\wedge dx_k$ and the constant $C_\Gamma \in (0,1]$ only depends on $\Gamma$.
\end{theorem}

\begin{rem*}
Assume $\Gamma=(\gamma_1,...,\gamma_k)$ is a simple closed multi-curve on $S_{g,n}$. If $(g,n)\neq (1,1),(2,0)$, then the coffecient $C_{\Gamma}$ in Theorem \ref{Mirz-for} is given as 
\begin{align}\label{e-coff}
C_\Gamma=\frac{2^{-M(\Gamma)}}{|\textnormal{Stab}(\Gamma):\cap_{i=1}^k\textnormal{Stab}^+(\gamma_i)|},
\end{align}
where 
$$M(\Gamma)=|\{i;\ \gamma_i\text{ separates off a one-handle from }S_{g,n}\}|,$$
$\textnormal{Stab}(\Gamma)$ is the stablizer of $\Gamma$ and $\textnormal{Stab}^{+}(\gamma_i)\ (1\leq i\leq k)$ is the subgroup of the mapping class group that fixes $\gamma_i$ and its orientation. For more details, one may refer to the footnote in \cite[Page 368]{Wri20}.
\end{rem*}


\subsection{Basic spectral theory}
For facts in this subsection, one may refer to \eg \cite{Ber-book, Bor-book} for more details. Let $X\in \sM_{g,n}$ where $n\geq 1$ be a non-compact hyperbolic surface of finite area, $\Delta$ be its Laplacian operator and $L^2(X)$ be the space of square-integrable functions on $X$. Starting from the domain $\mathcal{D}\overset{\text{def}}{=}\{f\in C_0^\infty(X); \ \textit{$f$ and $\Delta f \in L^2(X)$}\}$, the Friedrichs extension  uniquely extends $\Delta$ to be a non-negative self-adjoint operator on $L^2(X)$. The spectrum of $\Delta$ on $X$ consists of absolutely continuous spectrum $[\frac{1}{4},\infty)$ with multiplicity $n$ and possibly discrete spectrum in $[0,\infty)$. A number $\lambda\geq 0$ is called an \emph{eigenvalue} of $X$ if there exists a function $\phi\neq 0 \in L^2(X)$ such that
\[\Delta \phi=\lambda \cdot \phi.\]
Clearly $0$ is an eigenvalue and the constant function $\frac{1}{\sqrt{\mathrm{Area(X)}}}$ is the normalized $0$-eigenfunction. For a compact hyperbolic surface, its first non-zero eigenvalue can be determined by \emph{Rayleigh quotient} through a Min-Max Principle. For non-compact hyperbolic surface, it is hard to detect the existence of a non-zero eigenvalue of $X$ (see Question \ref{in-q-1}). The following fundamental result for spectral theory is well-known and important for the existence of first eigenvalue in Theorem \ref{mt-1}.
\bt \label{exist} (see \eg \cite[Theorem XIII.1]{RS-book})
Let $X$ be a non-compact hyperbolic surface of finite area. If
\[\mathrm{RayQ}(X)\overset{\text{def}}{=}\inf\limits_{f\neq 0\in L^2(X);  \int_X f=0} \frac{\int_X |\nabla f|^2}{\int_X f^2}<\frac{1}{4},\]
then $X$ has a non-zero first eigenvalue $\lambda_1(X)$ with $\lambda_1(X)=\mathrm{RayQ}(X)$.
\et

\noindent The second named author is grateful to Shiping Liu for pointing out the reference \cite{RS-book} to us. In this paper, we mainly study spectral gaps of Weil-Petersson random surfaces for large genus. For recent developments on spectral theory of Weil-Petersson random surfaces, one may see \cite{GMST21, Mirz13, Monk21, Ru22} and references therein for related topics.


\subsection{Probabilistic results}\label{sec-prob}
In this section, we recall two results from probability theory which will be applied later. For a probability space $(\Omega,\mathbb{P})$, we have
\begin{lemma}\label{l-cs}
For any non-negative integer-valued random variable $N$, the following inequality holds:
$$\mathbb{P}(N\geq 1)\geq\frac{\mathbb{E}[N]^2}{\mathbb{E}[N^2]}$$
where $\mathbb{E}[\cdot]$ is the expected value.
\end{lemma}
\begin{proof}
By Cauchy-Schwarz inequality, we have
\begin{align*}
\mathbb{E}[N]^2&=\mathbb{E}[N\cdot 1_{\{N\geq 1\}}]^2\\
&\leq \mathbb{E}[N^2]\cdot\mathbb{E}[1^2_{\{N\geq 1\}}]\\
&=\mathbb{E}[N^2]\cdot\mathbb{P}(N\geq 1).
\end{align*}
It follows that
$$\mathbb{P}(N\geq 1)\geq\frac{\mathbb{E}[N]^2}{\mathbb{E}[N^2]}$$ as desired.
\end{proof}

Let $\mathbb{N}$ be the set of all positive integers. For any $r\in \mathbb{N}$ and a random variable $N:\Omega\to\mathbb{N}\cup\{0\}$, we define the random variable
$$(N)_r\overset{\text{def}}{=}N(N-1)\cdots(N-r+1).$$
If the expected value $\mathbb{E}[(N)_r]$ of $(N)_r$ exists, then it is called the \emph{$r$-th factorial moment} of $N$. Recall that a non-negative integer-valued random variable $N$ is said to be \emph{Poisson distributed with mean $\lambda\in[0,\infty)$} if
$$\mathbb{P}(N=k)=\frac{\lambda^ke^{-\lambda}}{k!}\text{ for all }k\in\mathbb{N}\cup\{0\}.$$
It is known that a non-negative integer-valued random variable $N$ is Poisson distributed if and only if its $r$-th factorial moment is equal to $\lambda^n$ for all $r\in\mathbb{N}$. This can be generalized to the following form (see \eg \cite[Theorem 2.9]{MP19} or \cite[Theorem 1.21]{Bo-book}):
\begin{theorem}\label{t-pos}
Let $\{N_n\}_{n\geq 1}$ be a sequence of non-negative integer-valued random variables. If there exists a constant $\lambda>0$ such that for all $r\in\mathbb{N}$,
$$\lim\limits_{n\to\infty}\mathbb{E}\left[(N_n)_r\right]=\lambda^r,$$
then we have for all $k\in\mathbb{N} \cup\{0\}$,
$$\lim\limits_{n\to\infty}\mathbb{P}(N_n=k)=\frac{\lambda^k e^{-\lambda}}{k!}.$$
\end{theorem}


\subsection{Cheeger constant}\label{s-ch}
For any hyperbolic surface $X$, the Cheeger constant $h(X)$ of $X$ is defined as
$$h(X)\overset{\text{def}}{=}\inf\limits_{E\subset X}\frac{\ell(E)}{\min\{\textnormal{Area}(X_1),\textnormal{Area}(X_2)\}}$$
where $E$ runs over all rectifiable one-dimensional subsets of $X$ which divide $X$ into two disjoint components $X_1$ and $X_2$, $\ell(E)$ is the length of $E$, and $\textnormal{Area}(X_i)$ is the area of $X_i$ for $i=1,2$. For the case that $X$ is compact, let $\lambda_1(X)$ be the first eigenvalue of the Laplacian operator on $X$, the Cheeger's inequality \cite{Ch70} tells that $$\lambda_1(X)\geq \frac{1}{4}h^2(X).$$
For the case that $X$ is non-compact, since the existence of the first eigenvalue is not clear, we consider $\textnormal{RayQ}(X)$ instead of $\lambda_1(X)$. The following Cheeger's inequality for cusped hyperbolic surfaces will be applied in the proof of Theorem \ref{mt-4}.
\begin{lemma}[see \eg \textit{\cite[Page 228]{Bu82}}]\label{l-ch}
Assume $X$ is a complete hyperbolic surface of finite area, then
$$\textnormal{RayQ}(X)\geq \frac{1}{4}h^2(X).$$
\end{lemma}

For inverse Cheeger's inequality, Buser proved that
\begin{theorem}[\textit{\cite[Theorem 7.1]{Bu82}\label{t-Bu}}]
If $X$ is a complete, non-compact hyperbolic surface, then
$$\textnormal{RayQ}(X)\leq c\cdot h(X)$$
where $c>0$ is a universal constant.
\end{theorem}

\subsection{Geodesic Cheeger constant}\label{gch}
 Adams and Morgan investigated the isoperimetric problem on hyperbolic surfaces in \cite{AM99}. The following result is about the characterization of perimeter minimizers for cusped hyperbolic surfaces.
\begin{theorem}[\textit{\cite[Theorem 2.2]{AM99}}]\label{t-AM}
Let $X$ be a complete non-compact hyperbolic surface of finite area. For given area $0<A<\textnormal{area}(X)$, a perimeter-minimizing system of embedded rectifiable curves bounding a region $R$ of area $A$ consists of a set of curves of one of the following types:
\begin{enumerate}
\item horocycles around cusps;
\item geodesics or single ``neighboring curves" at constant distance from a geodesic.
\end{enumerate}
The total perimeter $L(A)$ of a minimizer of area $A$ satisfies that $L(A)\leq A$, and if $A<\pi$, then a minimizer consists of neighborhoods of an arbitrary collection of cusps bounded by horocycles, of total length $A$.
\end{theorem}
\noindent From Theorem \ref{t-AM}, we have $$h(X)=\inf\limits_{0<A\leq \frac{\textnormal{Area}(X)}{2}}\frac{L(A)}{A},$$
and $$L(A)=A \text{ for all } 0<A<\pi.$$
Hence the infimum above is realized at some value in $\left[\frac{\pi}{2}, \frac{\textnormal{Area}(X)}{2}\right]$. Denote by
$$\mathcal{S}(X)=\left(\alpha=\bigcup\limits_{i=1}^k\alpha_i;\
\begin{matrix}\text{For all }1\leq i\leq k,\ \alpha_i\text{ is a simple closed curve in $X$, }\\ \text{and }X\setminus\alpha=X_1\cup X_2,\text{ where $X_1$ and $X_2$ are}\\ \text{two disjoint connected components.} \end{matrix}\right)$$
and
$$\mathcal{SG}(X)=\left(\alpha=\bigsqcup\limits_{i=1}^k\alpha_i;\
\begin{matrix}\text{For all }1\leq i\leq k,\ \alpha_i\text{ is a simple closed geodesic in $X$, }\\ \text{and }X\setminus\alpha=X_1\cup X_2,\text{ where $X_1$ and $X_2$ are}\\ \text{two disjoint connected components.} \end{matrix}\right).$$
For any $\alpha\in\mathcal{S}(X)$ such that $X\setminus\alpha=X_1\cup X_2$, denote by 
\be\label{def-h}
\mathcal{H}(\alpha)=\frac{\ell(\alpha)}{\min\{\textnormal{Area}(X_1),\textnormal{Area}(X_2)\}}
\ene
where $\ell(\alpha)$ is the hyperbolic length of $\alpha$ and $\textnormal{Area}(X_i)$ is the area of $X_i$ for $i=1,2$. Due to S. T. Yau (see \eg \cite[Theorem 8.3.6]{Buser-book}) and Theorem \ref{t-AM}, one may check that
\begin{align}\label{dh-1}
h(X)=\inf\limits_{\alpha\in\mathcal{S}(X)}\mathcal{H}(\alpha).
\end{align}
Mirzakhani defined \emph{the geodesic Cheeger constant} in \cite{Mirz13} as
\[ H(X)\overset{\textnormal{def}}{=}\min\limits_{\alpha\in\mathcal{SG}(X)}\mathcal{H}(\alpha). \]
Based on Theorem \ref{t-AM}, Mirzakhani \cite{Mirz13} observed the following connection between $H(X)$ and $h(X)$. Assume $\alpha=\bigcup_{i=1}^k \alpha_i\in\mathcal{S}(X)$ and $\alpha'=\bigcup_{i=1}^k \alpha_i'\in\mathcal{SG}(X)$ such that $\alpha_i$ is a neighboring curve at a constant distance from a simple closed geodesic $\alpha_i'$ for all $1\leq i\leq k$. Then a simple computation yields that
\begin{align}\label{dh-3}
\mathcal{H}(\alpha)\geq\frac{\mathcal{H}(\alpha')}{1+\mathcal{H}(\alpha')}.
\end{align}
This implies that
\begin{proposition}[\textit{\cite[Proposition 4.7]{Mirz13}}]\label{lh-mirz}
Let $X$ be a complete non-compact hyperbolic surface of finite area, then
$$\frac{H(X)}{1+H(X)}\leq h(X)\leq H(X).$$
\end{proposition}
\begin{rem*}
Mirzakhani stated \cite[Proposition 4.7]{Mirz13} only for closed hyperbolic surfaces. Actually its proof also leads to the case of cusped hyperbolic surfaces.
\end{rem*}

\section{WP volume}\label{sec-wp}
In this section, we first recall several estimates about the Weil-Petersson volume $V_{g,n}$ of moduli space $\mathcal{M}_{g,n}$. Then we provide certain useful bounds on $V_{g,n}$ which will be applied in the proofs of Theorem \ref{mt-1}, \ref{mt-3} and \ref{mt-4}.

\subsection{Asymptotic properties on $V_{g,n}$}
In this subsection, we first recall several results about $V_{g,n}$, and then make two useful estimations. One may also see \cite{Penner92, ST01, Mirz07a, Mirz13, MZ15, AM21,Agg20} and references therein for the asymptotic behavior of $V_{g,n}$ and its deep connection to the intersection theory of $\M_{g,n}$.

We always let $V_{0,3}(x,y,z)=1$ (see \cite[Table $1$]{Mirz07a}). The following lemma is due to Mirzakhani.

\begin{lemma}[\textit{\cite[Lemma 3.2]{Mirz13}}]\label{l-mirz}
For any $g,n\geq 0$ with $2g-2+n>0$, we have
\begin{enumerate}
  \item
\begin{equation*}
V_{g,n} \leq V_{g,n}(x_1,\cdots,x_n) \leq e^{\frac{x_1+\cdots+x_n}{2}} V_{g,n},
\end{equation*}

  \item
\begin{equation*}
V_{g-1,n+4} \leq V_{g,n+2}
\end{equation*}
and
\begin{equation*}
\frac{V_{g,n+1}}{(2g-2+n)V_{g,n}}\asymp 1.
\end{equation*}
\end{enumerate}
\end{lemma}

The following two results are due to Mirzakhani-Zograf.
\begin{lemma}[\textit{\cite[Lemma 5.1]{MZ15}}]\label{l-MZ}
There exists a universal constant $c_2>0$ such that for any $g,n\geq 0$ with $2g-2+n\geq 1$, the following inequality holds:
$$\left|\frac{(2g-2+n)V_{g,n}}{V_{g,n+1}}-\frac{1}{4\pi^2}\right|\leq \frac{c_2n}{2g-2+n}.$$
\end{lemma}

As a direct consequence, we have
\begin{lemma}\label{l-cor-3}
 If $n(g)$ satisfies that $\lim\limits_{g\to\infty}\frac{n(g)}{g}=0,$
then we have
$$\lim\limits_{g\to\infty}\frac{V_{g,n(g)}^2}{V_{g,n(g)-1}V_{g,n(g)+1}}=1.$$
\end{lemma}

Mirzakhani and Zograf also proved the following aysmptotic behavior of  $V_{g,n}$ for large $g$.
\begin{theorem}[\textit{\cite[Theorem 1.2]{MZ15}}] \label{t1-MZ}
There exists a universal constant $C>0$ such that for any given $n\geq0$,
\begin{equation*}
V_{g,n} =  \frac{C}{\sqrt{g}} (2g-3+n)! \cdot (4\pi^2)^{2g-3+n} \cdot \left(1+O\left(\frac{1}{g}\right) \right)
\end{equation*}
as $g\rightarrow\infty$. The implied constant is related to $n$ and independent of $g$.
\end{theorem}

\begin{rem*}
It is conjectured by Zograf in \cite{Zograf08} that $C=\frac{1}{\sqrt\pi}$.
\end{rem*}

Now we apply the results above to get several useful new bounds on $V_{g,n}$. Our first one is as follows.
\begin{lemma}\label{l-cor-1}
For any non-negative integers $g,n$ such that $2g-2+n\geq 2$,
$$V_{g,n}\prec\frac{(2g-3+n)! \cdot (4\pi^2)^{2g-3+n}}{\sqrt{2g-2+n}}.$$
\end{lemma}
\begin{proof}
From Part $(2)$ of Lemma \ref{l-mirz} and the assumption that $2g-2+n\geq 2$, there exist $g^\prime>0$ and $0\leq n^\prime\leq 3$ such that
$$2g-2+n=2g^\prime-2+n^\prime\text{ and }V_{g,n}\leq V_{g^\prime,n^\prime}.$$
Apply Theorem \ref{t1-MZ} to $V_{g^\prime,n^\prime}$, we have
\begin{align*}
V_{g,n}&\leq V_{g^\prime,n^\prime}
\prec \frac{(2g^\prime-3+n^\prime)!\cdot (4\pi^2)^{2g^\prime-3+n^\prime}}{\sqrt {g^\prime}}\\
&\prec\frac{(2g-3+n)!\cdot (4\pi^2)^{2g-3+n}}{\sqrt{2g-2+n}}
\end{align*}
where the last inequality holds since $2g-2+n\prec g^\prime$. The proof is complete.
\end{proof}

The following lemma is implicitly contained in the proof of \cite[Theorem 1.8]{MZ15}. We prove it here for completeness.
\begin{lemma}\label{l-cor-2}
Assume $n(g)\prec\sqrt{g}$, then we have
$$V_{g,n(g)}\asymp \frac{(2g-3+n(g))! \cdot(4\pi^2)^{2g-3+n(g)}}{\sqrt g}.$$
\end{lemma}
\begin{proof}
It suffices to consider the case when $g$ is large. From Lemma \ref{l-MZ}, there exists a universal constant $c_3>0$ such that for any $0\leq n\leq n(g)$ and large enough $g$,
\begin{align}\label{lc-e-0}
\frac{1}{4\pi^2}\left(1-\frac{c_3n}{g}\right)\leq\frac{(2g-2+n)V_{g,n}}{V_{g,n+1}}\leq \frac{1}{4\pi^2}\left(1+\frac{c_3n}{g}\right).
\end{align}
Denote by
$$\mu_1=\prod\limits_{i=1}^{n(g)}\left(1-\frac{c_3i}{g}\right)\text{ and }\mu_2=\prod\limits_{i=1}^{n(g)}\left(1+\frac{c_3i}{g}\right).$$
It is clear that for $x>0$ small enough, $\log(1+x)\leq x$ and $\log(1-x)\geq -2x$. Since $n(g)\prec\sqrt g$, it follows that for $g$ large enough,
$$\log \mu_1=\sum\limits_{i=1}^{n(g)}\log\left(1-\frac{c_3i}{g}\right)\geq \sum\limits_{i=1}^{n(g)}-\frac{2c_3 i}{g}\asymp -\frac{n(g)^2}{g}\succ -1$$
and
\begin{align*}
\log \mu_2=\sum\limits_{i=1}^{n(g)}\log\left(1+\frac{c_3i}{g}\right)\leq \sum\limits_{i=1}^{n(g)}\frac{c_3i}{g}\asymp \frac{n(g)^2}{g}\prec 1.
\end{align*}
Hence one may assume that there exists a universal constant $c_4>1$ such that
\begin{align}\label{lc-e-1}
\frac{1}{c_4}<\mu_1<\mu_2<c_4.
\end{align}
For any $g,n\geq 0$ with $2g-2+n\geq 2$, denote by
$$C_{g,n}=\frac{(2g-3+n)!\cdot(4\pi^2)^{2g-3+n}}{\sqrt g}.$$
Then we have
$$\frac{V_{g,n(g)}}{C_{g,n(g)}}=\prod\limits_{i=0}^{n(g)-1}\frac{ V_{g,n(g)-i}}{4\pi^2(2g-3+n(g)-i)V_{g,n(g)-i-1}}
\cdot\frac{V_{g,0}}{C_{g,0}}.$$
Together with \eqref{lc-e-0} and \eqref{lc-e-1}, it follows that
$$\frac{1}{c_4}\cdot\frac{V_{g,0}}{C_{g,0}}\leq\frac{1}{\mu_2}\cdot\frac{V_{g,0}}{C_{g,0}}\leq\frac{V_{g,n(g)}}{C_{g,n(g)}}\leq \frac{1}{\mu_1}\cdot \frac{V_{g,0}}{C_{g,0}}\leq c_4\cdot\frac{V_{g,0}}{C_{g,0}}.$$
Then the proof is complete by applying Theorem \ref{t1-MZ}.
\end{proof}

\subsection{A useful bound for WP volume}\label{sub-I}
For $m\in\mathbb{N}$ and non-negative integers $g,n(g)$ with \be\chi\overset{\text{def}}{=}2g-2+n(g)\geq 1,\ene we denote by $\mathcal{I}_m$ the set consisting of all 4-tuples of non-negative integers $(g_1,n_1,g_2,n_2)$ satisfying the following conditions:
\begin{enumerate}[(a)]
\item $2g_1-2+n_1+2g_2-2+n_2=\chi$.
\item $m=2g_1-2+n_1\leq 2g_2-2+n_2$.
\item $n(g)-n_1\leq n_2\leq n(g)+n_1$, $n_1,n_2\geq 1$.
\end{enumerate}
From condition $(b)$, there are at most $m+2$ different choices for $(g_1,n_1)$. Since $n_1\leq m+2$, together with condition $(c)$, there are at most $2m+5$ different choices for $n_2$. Recall that from condition $(b)$, $g_1$ is uniquely determined by $n_1$ and $m$. Thus it follows that
 \begin{align}\label{pv2-e00}
 \left|\mathcal{I}_m\right|\leq 2(m+3)^2.
 \end{align}

\begin{lemma}\label{l-ratio}
Assume $n(g)\prec\sqrt{g}$. Then for any pair $(g_1,n_1,g_2,n_2)\in\mathcal{I}_m$,
\[\frac{V_{g_1,n_1}V_{g_2,n_2}}{V_{g,n(g)}}\prec \frac{m^{m}(\chi-m)^{\chi-m}}{\chi^\chi}.\]
\end{lemma}
\bp
Since $2m\leq \chi$ and $n(g)\prec\sqrt{g}$,
\be\label{i-r-1}
(\chi-m)\asymp \chi \asymp g.
\ene
Then by Lemma \ref{l-cor-2} we have
\be\label{i-r-2}
V_{g,n(g)}\asymp \frac{(\chi-1)!\cdot (4\pi^2)^{\chi}}{\sqrt{\chi}}.
\ene
By Lemma \ref{l-cor-1} we also have
\be\label{i-r-3}
V_{g_1,n_1}\prec \frac{(m-1)!\cdot (4\pi^2)^{m}}{\sqrt{m}} \textit{\ and \ } V_{g_2,n_2}\prec \frac{(\chi-m-1)!\cdot (4\pi^2)^{\chi-m}}{\sqrt{\chi-m}}.
\ene
Combining \eqref{i-r-1}, \eqref{i-r-2} and \eqref{i-r-3}, we have
\be\label{i-r-4}
\frac{V_{g_1,n_1}V_{g_2,n_2}}{V_{g,n(g)}}\prec\frac{(m-1)!(\chi-m-1)!}{(\chi-1)!}\cdot\frac{\sqrt{g}}{\sqrt{m}\sqrt{\chi-m}}\prec\frac{m!(\chi-m)!}{\chi!}\cdot\frac{1}{m}.
\ene
Recall that the Stirling's formula says that
$$k!\sim\sqrt{2\pi k}\left(\frac{k}{e}\right)^k.$$
Plug it into \eqref{i-r-4}, we get
\[\frac{V_{g_1,n_1}V_{g_2,n_2}}{V_{g,n(g)}}\prec \frac{m^{m}(\chi-m)^{\chi-m}}{\chi^\chi}\cdot\frac{\sqrt{\chi-m}}{\sqrt m\sqrt\chi}\prec \frac{m^{m}(\chi-m)^{\chi-m}}{\chi^\chi}\]
as desired.
\ep

Now we prove the following result which is useful in the proof of Theorem \ref{mt-4}.
\begin{proposition}\label{p-vol-2}
Assume $n(g)\prec \sqrt g$. For $0<u<\frac{\log 2}{2\pi}$ and $m\geq 2\in \mathbb{N}$, set $$L_m=2\pi mu+3\cdot (2\pi mu)^{\frac{2}{3}}.$$ Then we have
$$\frac{1}{V_{g,n(g)}}\sum\limits_{m=2}^{\lfloor \frac{\chi}{2}\rfloor}\sum\limits_{(g_1,n_1,g_2,n_2)\in\mathcal{I}_m}C(m,n(g))V_{g_1,n_1}V_{g_2,n_2}e^{L_m} \prec \frac{1}{\sqrt g},$$
where $C(m,n(g))=\max\limits_{0\leq i\leq m+1} \binom{n(g)}{i}.$
 \end{proposition}
\begin{proof}
Since $n(g)\prec\sqrt g$, one may assume $n(g)\leq a\sqrt g$ for some $a>0$. It suffices to consider for large enough $g$.
 For all $2\leq m\leq\lfloor\frac{\chi}{2}\rfloor$, set
$$A_m=\frac{1}{V_{g,n(g)}}\sum\limits_{(g_1,n_1,g_2,n_2)\in\mathcal{I}_m}C(m,n(g))V_{g_1,n_1}V_{g_2,n_2}e^{L_m}.$$
Since $u<\frac{\log 2}{2\pi}$, it follows that for $m$ large enough
\begin{align*}
e^{\frac{L_m}{m}}=e^{2\pi u+3(2\pi u)^{\frac{2}{3}}m^{-\frac{1}{3}}}<2-\delta
\end{align*}
for some constant $\delta>0$ only depending on $u$. So we have for all $m\geq 2$,
\begin{align*}
e^{L_m}\prec (2-\delta)^m.
\end{align*}
Then from \eqref{pv2-e00} and Lemma \ref{l-ratio} we have
\begin{align}\label{pv2-e1}
A_m&\prec C(m,n(g))\frac{m^m(\chi-m)^{\chi-m}}{\chi^\chi}e^{L_m}m^2\\
&\prec C(m,n(g))\left(\frac{m(2-\delta)}{\chi}\right)^m m^2.\nonumber
\end{align}
It is clear that for any $m\geq 2$, we also have
\begin{align}\label{pv2-e2}
\max\limits_{0\leq i\leq m+1}\binom{n(g)}{i}=\max\limits_{0\leq i\leq m+1}\frac{n(g)!}{i!\cdot (n(g)-i)!}\leq\min\left\{n(g)^{m+1},2^{n(g)}\right\}.
\end{align}
\

\noindent Next we split the argument into four different cases.
\

\noindent \underline{Case I: $g^{\frac{3}{4}}\leq m\leq\lfloor\frac{\chi}{2}\rfloor$.} From \eqref{pv2-e1} and \eqref{pv2-e2} we have
\begin{align*}
A_m&\prec C(m,n(g))\left(\frac{m(2-\delta)}{\chi}\right)^m m^2\\
&\prec 2^{\left\lfloor a\sqrt g\right\rfloor}\left(\frac{2-\delta}{2}\right)^{\left\lfloor g^{\frac{3}{4}}\right\rfloor}g^2\prec\frac{1}{g^2}.\nonumber
\end{align*}
Then it follows that
\begin{align}\label{pv2-e3}
\sum\limits_{m=\left\lfloor g^{\frac{3}{4}}\right\rfloor+1}^{\lfloor\frac{\chi}{2}\rfloor}A_m \prec \frac{1}{g}.
\end{align}
\underline{Case II: $\frac{\sqrt g}{4ea}\leq m\leq g^{\frac{3}{4}}$.} From \eqref{pv2-e1} and \eqref{pv2-e2} we have
\begin{align*}
A_m&\prec C(m,n(g))\left(\frac{m(2-\delta)}{\chi}\right)^m m^2\\
&\prec 2^{\left\lfloor a\sqrt g\right\rfloor}\left(\frac{2}{g^{\frac{1}{4}}}\right)^{\left\lfloor \frac{\sqrt g}{4ea}\right\rfloor}m^2\nonumber\\
&\prec\left(\frac{2^{4ea^2+1}}{g^{\frac{1}{4}}}\right)^{\left\lfloor \frac{\sqrt g}{4ea}\right\rfloor}\cdot g^3
\prec\frac{1}{g^2}.\nonumber
\end{align*}
Then it follows that
\begin{align}\label{pv2-e4}
\sum\limits_{m=\lfloor \frac{\sqrt g}{4ea}\rfloor}^{\lfloor g^{\frac{3}{4}}\rfloor}A_m \prec\frac{1}{g}.
\end{align}
\underline{Case III: $4\log g\leq m\leq \frac{\sqrt g}{4ea}$.} For this case we have $n(g)\cdot m \leq a \sqrt{g} \cdot \frac{\sqrt g}{4ea}=\frac{g}{4e}$. Then from \eqref{pv2-e1} and \eqref{pv2-e2} we have
\begin{align*}
A_m&\prec n(g)^{m+1}\left(\frac{m(2-\delta)}{\chi}\right)^m m^2\\
   &\prec \left(\frac{2n(g)m}{\chi}\right)^m n(g)m^2\\
   &\prec \left(\frac{1}{e}\right)^m n(g)\cdot g\prec \left(\frac{1}{e}\right)^m  g^{\frac{3}{2}}.
\end{align*}
Then it follows that
\begin{align}\label{pv2-e5}
\sum\limits_{m=\left\lfloor 4\log g\right\rfloor}^{\left\lfloor \frac{\sqrt g}{4ea}\right\rfloor}A_m\prec\sum\limits_{m=\left\lfloor 4\log g\right\rfloor}^{\left\lfloor\frac{\sqrt g}{4ea}\right\rfloor}\left(\frac{1}{e}\right)^m g^{\frac{3}{2}} \prec e^{-4\log g}\cdot \sqrt{g}\cdot g^{\frac{3}{2}} \prec \frac{1}{g}.
\end{align}
\underline{Case IV: $2\leq m\leq 4\log g$.} From \eqref{pv2-e1} and \eqref{pv2-e2} we have
\begin{align}\label{pv2-e6}
\sum\limits_{m=2}^{\left\lfloor 4\log g\right\rfloor}A_m&\prec\sum\limits_{m=2}^{\left\lfloor 4\log g\right\rfloor} n(g)^{m+1}\left(\frac{m(2-\delta)}{\chi}\right)^m m^2\\
&\prec\sum\limits_{m=2}^{\left\lfloor 4\log g\right\rfloor}\left(\frac{2am}{\sqrt g}\right)^m n(g)m^2 \nonumber \\
&\leq \frac{64a^2n(g)}{g}+\sum\limits_{m=3}^{\left\lfloor 4\log g\right\rfloor}\left(\frac{8a\log g}{\sqrt g}\right)^3n(g)(4\log g)^2\prec \frac{1}{\sqrt g}. \nonumber
\end{align}

\noindent Then the conclusion follows from \eqref{pv2-e3}, \eqref{pv2-e4}, \eqref{pv2-e5} and \eqref{pv2-e6}.
\end{proof}

\subsection{Intersection number and WP volume}\label{sub-e} Recall that on $\overline{\M}_{g,n}$ the compactified moduli space of surfaces of genus $g$ with $n$ punctures, there are $n$ tautological line bundles $\{\mathcal{L}_i\}_{i=1}^n$. Let $\psi_i =c_1(\mathcal{L}_i) \in H_2(\overline{\M}_{g,n},\mathbb{Q})$ be the first Chern class of $\mathcal{L}_i$ and $\omega$ be the Weil-Petersson symplectic form on $\overline{\M}_{g,n}$. For any $\textbf{d}=(d_1,\cdots,d_n)$ with each $d_i\in\Z_{\geq 0}$ and $|\textbf{d}|=d_1+\cdots+d_n\leq 3g-3+n$, the \emph{intersection number} $\left[\tau_{d_1}\cdots\tau_{d_n}\right]_{g,n}$ is defined as
\begin{equation*}
\left[\tau_{d_1}\cdots\tau_{d_n}\right]_{g,n} =\left[\prod\limits_{i=1}^n\tau_{d_i}\right]_{g,n}\overset{\text{def}}{=} \frac{\prod_{i=1}^n 2^{2d_i}(2d_i+1)!!}{(3g-3+n-|d|)!} \int_{\overline{\M}_{g,n}} \psi_1^{d_1}\cdots \psi_n^{d_n} \omega^{3g-3+n-|d|}.
\end{equation*}
In particular, $\left[\tau_{0}\cdots\tau_{0}\right]_{g,n}=\Vol(\sM_{g,n})$, which is also denoted by $V_{g,n}$.

In this subsection we use the same notations as in \cite{Mirz13, MZ15}.
Recall \cite[(3.1)]{Mirz13} tells that for $x_i\geq 0 \ (1\leq i\leq n)$
\begin{align}\label{e-vol}
V_{g,n}(2x_1,...,2x_n)=\sum\limits_{|\textbf{d}|\leq 3g-3+n}\left[\prod\limits_{i=1}^n\tau_{d_i}\right]_{g,n}\cdot\prod\limits_{i=1}^n
\frac{x_i^{2d_i}}{(2d_i+1)!}.
\end{align}
 To bound $V_{g,n}(2x_1,...,2x_n)$, it suffices to bound the intersection numbers $\left[\prod\limits_{i=1}^n\tau_{d_i}\right]_{g,n}$.

As in \cite{Mirz13} we let $a_0=\frac{1}{2}$ and for $i\geq 1$,
$$a_i=\zeta(2i)(1-2^{1-2i}),$$
where $\zeta(\cdot)$ is the Riemann zeta function.
\begin{lemma}[\textit{\cite[Lemma 2.2]{MZ15}}]\label{l-ai}
The sequence $\{a_i\}_{i=0}^\infty$ is increasing. Moreover,
\begin{enumerate}
\item $$\lim\limits_{i\to\infty}a_i=1.$$
\item There exists $C>1$ such that
$$\frac{1}{C\cdot 2^{2i}}<a_{i+1}-a_i<\frac{C}{2^{2i}}.$$
\end{enumerate}
\end{lemma}
\noindent As a direct consequence, for any $k \in\mathbb{N}$,
\begin{align}\label{e-ap-1}
\sum\limits_{L=1}^\infty (L+k)(a_{L+k}-a_L)\prec k.
\end{align}
For any $\textbf{d}=(d_1,\cdots,d_n)$, let $d_0=3g-3+n-|\textbf{d}|$, the Recursive Formula (see \cite[Section 3.1]{Mirz13}) says that
$$\left[\prod\limits_{i=1}^n\tau_{d_i}\right]_{g,n}=\left(\sum\limits_{j=2}^n\mathcal{A}_{\textbf{d}}^j\right)
+\mathcal{B}_{\textbf{d}}+\mathcal{C}_{\textbf{d}},$$
where
$$\mathcal{A}_{\textbf{d}}^j=8\sum\limits_{L=0}^{d_0}(2d_j+1)a_L\left[\tau_{d_1+d_j+L-1}\cdot\prod\limits_{i\neq1,j}\tau_{d_i}\right]_{g,n-1},$$
$$\mathcal{B}_{\textbf{d}}=16\sum\limits_{L=0}^{d_0}\sum\limits_{k_1+k_2=L+d_1-2}a_L\left[\tau_{k_1}\tau_{k_2}\prod\limits_{i\neq 1}\tau_{d_i}\right]_{g,n+1},$$
and
\begin{align*}
\mathcal{C}_{\textbf{d}}&=16\sum\limits_{\substack{I\sqcup J=\{2,...,n\}\\ 0\leq g^\prime\leq g}}\sum\limits_{L=0}^{d_0}\sum\limits_{k_1+k_2=L+d_1-2}a_L
\left[\tau_{k_1}\prod\limits_{i\in I}\tau_{d_i}\right]_{g^\prime,|I|+1}\\
&\times\left[\tau_{k_2}\prod\limits_{i\in J}\tau_{d_i}\right]_{g-g^\prime,|J|+1}.
\end{align*}
Through comparing the coefficients of $\mathcal{A}_{\textbf{d}},\mathcal{B}_{\textbf{d}}$ and $\mathcal{C}_{\textbf{d}}$, it is not hard to see that (see \cite[Equation 3.6]{MZ15})
\begin{align}\label{e-ap-2}
\left[\prod\limits_{i=1}^n\tau_{d_i}\right]_{g,n}\leq V_{g,n}.
\end{align}

Now we prove the following result.
\begin{lemma}\label{l-ap}
Assume $n\prec\sqrt g$, and set $\textbf{d}=(d_1,...,d_n)$ and $\textbf{d}^\prime=(d_1^\prime,...,d_n^\prime)=(0,d_2,...,d_n)$. Then there exists a constant $C>0$ such that
$$0\leq\left[\prod\limits_{i=1}^n\tau_{d_i^\prime}\right]_{g,n}-
\left[\prod\limits_{i=1}^n\tau_{d_i}\right]_{g,n}\leq
\frac{C|\textbf{d}|(|\textbf{d}|+n)}{g}\cdot V_{g,n}.$$
\end{lemma}
\begin{proof}
Since the proof is similar to the proofs of \cite[Equation 3.23]{Mirz13} and \cite[Lemma 5.1]{MZ15}, we only give an outline here. For the LHS inequality, it directly follows by comparing the terms $\mathcal{A}_{\textbf{d}}^j$, $\mathcal{B}_{\textbf{d}}^j$ and $\mathcal{C}_{\textbf{d}}^j$ with $\mathcal{A}_{\textbf{d}^\prime}^j$, $\mathcal{B}_{\textbf{d}^\prime}^j$ and $\mathcal{C}_{\textbf{d}^\prime}^j$ respectively in the Recursive Formula above. For the RHS inequality, $\mathrm{WLOG}$ one may first assume that $d_1> 0$.  Then compare the terms $\mathcal{A},\mathcal{B}$ and $\mathcal{C}$ respectively, it follows from Lemma \ref{l-ai} and \eqref{e-ap-2} that for any fixed $2\leq j\leq n$,
\begin{align*}
&\ \ \ \ \mathcal{A}_{\textbf{d}^\prime}^j-\mathcal{A}_{\textbf{d}}^j\\
&=8\sum\limits_{L=0}^{d_0}(2d_j+1)a_L\left[\tau_{d_j+L-1}\cdot\prod\limits_{i\neq1,j}\tau_{d_i}\right]_{g,n-1}\\
&-8\sum\limits_{L=0}^{d_0}(2d_j+1)a_L\left[\tau_{d_1+d_j+L-1}\cdot\prod\limits_{i\neq1,j}\tau_{d_i}\right]_{g,n-1}\\
&=8(2d_j+1)\sum\limits_{t=d_j-1}^{d_0+d_j-1}a_{t-(d_j-1)}\left[\tau_t\prod\limits_{i\neq1,j}\tau_{d_i}\right]_{g,n-1}\\
&-8(2d_j+1)\sum\limits_{t=d_1+d_j-1}^{d_0+d_1+d_j-1}a_{t-(d_1+d_j-1)}\left[\tau_t\prod\limits_{i\neq1,j}\tau_{d_i}\right]_{g,n-1}\\
&\leq 8d_1(2d_j+1)V_{g,n-1}+\sum\limits_{t=d_1+d_j-1}^{d_0+d_j-1}(a_{t-(d_j-1)}-a_{t-(d_1+d_j-1)})V_{g,n-1}\\
&\prec\frac{d_1(2d_j+1)}{g}\cdot V_{g,n} \prec\frac{\max_{1\leq i \leq n}d_i \cdot (2d_j+1)}{g}\cdot V_{g,n}.
\end{align*}
This leads to
\[\sum\limits_{j=2}^n(\mathcal{A}_{\textbf{d}^\prime}^j-\mathcal{A}_{\textbf{d}}^j)\prec \frac{\max_{1\leq i \leq n}d_i \cdot(|\textbf d|+n)}{g}\cdot V_{g,n}\prec \frac{|\textbf{d}|(|\textbf{d}|+n)}{g}\cdot V_{g,n}.\]

\noindent Similarly, by applying \eqref{e-ap-1} and \eqref{e-ap-2}, we have
\begin{align*}
\mathcal{B}_{\textbf{d}^\prime}-\mathcal{B}_{\textbf{d}}\prec\frac{|\textbf{d}|^2}{g}\cdot V_{g,n}.
\end{align*}

\noindent For the term $\mathcal{C}$, first by Lemma \ref{l-cor-1} and \ref{l-cor-2} it is not hard to see that
$$\sum\limits_{(g_1,n_1,g_2,n_2)} \binom{n-1}{n_1}\cdot\frac{V_{g_1,n_1+1}V_{g_2,n_2+1}}{V_{g,n}}=O\left(\frac{n^2}{g^2}\right)$$
where $(g_1,n_1,g_2,n_2)$ runs over all pairs of non-negative integers such that
$$g_1+g_2=g\text{ and }n_1+n_2=n-1.$$
Then similar as the proof for $\mathcal{A}_{\textbf{d}^\prime}^j-\mathcal{A}_{\textbf{d}}^j$, by applying \eqref{e-ap-1} and \eqref{e-ap-2} one may have
$$\mathcal{C}_{\textbf{d}^\prime}-\mathcal{C}_{\textbf{d}}\prec\frac{n^2|\textbf{d}|^2}{g^2}\cdot V_{g,n}\prec\frac{|\textbf{d}|^2}{g}\cdot V_{g,n}.$$
Then the conclusion follows by taking a summation over all $\mathcal{A}_{\textbf{d}}^j,\ \mathcal{B}_{\textbf{d}}$ and $\mathcal{C}_{\textbf{d}}$.
\end{proof}
\begin{lemma}\label{l-int}
 Assume $n(g)\prec\sqrt{g}$ and $k\in\mathbb{N}$, then there exists a constant $C>0$ such that for any $\textbf{d}=(d_1,...,d_k)$ with $d_i\geq 0$ and $|\textbf{d}|=d_1+...+d_k$,
$$0\leq 1-\frac{\left[\prod\limits_{i=1}^k\tau_{d_i}\cdot\tau_{0}^{n-k}\right]_{g,n}}{V_{g,n}}\leq \frac{Ck\left(|\textbf{d}|^2+n|\textbf{d}|\right)}{g}.$$
\end{lemma}
\begin{proof}
The $\mathrm{LHS}$ inequality clearly follows from \eqref{e-ap-2}. For the $\mathrm{RHS}$ inequality, it follows from Lemma \ref{l-ap} that for any $1 \leq j\leq k-1$,
\begin{align}\label{e-ap-3}
\left[\prod\limits_{i=1}^j\tau_{d_i}\cdot\tau_0^{n-j}\right]_{g,n}-
\left[\prod\limits_{i=1}^{j+1}\tau_{d_i}\cdot\tau_0^{n-j-1}\right]_{g,n}\leq
\frac{C|\textbf{d}|(|\textbf{d}|+n)}{g}\cdot V_{g,n},
\end{align}
and
\begin{align}\label{e-ap-4}
V_{g,n}-\left[\tau_{d_1}\cdot\tau_0^{n-1}\right]_{g,n}\leq\frac{C|\textbf{d}|(|\textbf{d}|+n)}{g}\cdot V_{g,n}.
\end{align}
Then the conclusion follows by taking a summation of \eqref{e-ap-3} over all $1\leq j\leq k-1$ and \eqref{e-ap-4}.
\end{proof}

The following proposition is essentially due to \cite[Proposition 31]{MP19} and \cite[Lemma 20]{NWX20}.
\begin{proposition}\label{p-vol-1}
Assume $n(g)\prec\sqrt{g}$ and $k\in\mathbb{N}$, then for any $x_i\geq 0$,
$$\prod\limits_{i=1}^k\frac{\sinh x_i}{x_i}\cdot \left(1+O\left(\frac{nk^2\sum_{i=1}^k x_i^2}{g}\right)\right)=\frac{V_{g,n}(2x_1,...,2x_k,0,...,0)}{V_{g,n}}\leq \prod\limits_{i=1}^k\frac{\sinh x_i}{x_i},$$
where the implied constant is independent of $g,n,k$ and $x_i$.
\end{proposition}
\begin{proof}
The proof is identical to the proof of \cite[Lemma 20]{NWX20}, if one applies Lemma \ref{l-int} in place of \cite[p.286]{Mirz13}. We leave it to interested readers.
\end{proof}

\section{Two types of separating simple closed multi-curves}\label{multi-curve}
In this section, we first introduce two special types of separating simple closed multi-curves, and then calculate the expected values of related random variables. We always assume $g,\ n\geq 0$ with $2g-2+n\geq 1$ and $S_{g,n}$ is an orientable surface of genus $g$ with $n$ punctures denoted by $\{p_i\}_{i=1}^n$.
\subsection{Separating simple closed multi-curves}
Assume  $\Gamma=\sum_{i=1}^k\gamma_i$ is a simple closed multi-curve on $S_{g,n}$, where the $\gamma^\prime_i$s are pairwisely disjoint homotopy classes of non-trivial and non-peripheral simple closed curves on $S_{g,n}$. Recall that for any element $h\in\textnormal{Mod}_{g,n}$, we have $h(p_j)=p_j$ for all $1\leq j \leq n$. It follows that the orbit $\mathcal{O}_\Gamma=\textnormal{Mod}_{g,n}\cdot\Gamma$ does not only depend on the topology of $S_{g,n}\setminus\Gamma$, but also depends on the way that $\Gamma$ separates the set of all punctures of $S_{g,n}$. For $1\leq k\leq g+1$, we consider a pair $\mathfrak{I}=(J_1,J_2,g_1,g_2,k)$ which satisfies the following properties:
 \begin{enumerate}[(a)]
\item $J_1$ and $J_2$ are disjoint subsets (maybe empty) of $\{1,2,...,n\}$ such that
$$\{1,2,...,n\}=J_1\sqcup J_2.$$
\item $g_1$ and $g_2$ are non-negative integers such that
$$g=g_1+g_2+k-1.$$
\item $1\leq 2g_1-2+|J_1|+k\leq 2g_2-2+|J_2|+k$ where $|J_i|$ is the cardinality of $J_i$.
\end{enumerate}
Then $\mathfrak{I}=(J_1,J_2,g_1,g_2,k)$ induces a simple closed multi-curve denoted by $\Gamma(\mathfrak{I})=\sum_{i=1}^k\gamma_i$ in $S_{g,n}$ such that
$$S_{g,n}\setminus\Gamma(\mathfrak{I})\cong S_{g_1,|J_1|+k}\cup S_{g_2,|J_2|+k},$$
where the set $\{p_j\}_{j\in J_1}$ of punctures is contained in the part $S_{g_1,|J_1|+k}$, and $\{p_j\}_{j\in J_2}$ is contained in the part $S_{g_2,|J_2|+k}$.
\begin{remark}\label{r-1}
\begin{enumerate}
\item Assume $\mathfrak{I}=(J_1,J_2,g_1,g_2,k)$ and $\mathfrak{I}^\prime=(J_1^\prime,J_2^\prime,g_1^\prime,g_2^\prime,k^\prime)$ are two pairs, then $\Gamma(\mathfrak{I})$ and $\Gamma(\mathfrak{I}^\prime)$ are in the same orbit under $\textnormal{Mod}_{g,n}$-action if and only if $k=k^\prime$ and $\left\{(J_1,g_1),(J_2,g_2)\right\}=\left\{(J_1^\prime,g_1^\prime),(J_2^\prime,g_2^\prime)\right\}$.

\item For $\mathfrak{I}=(J_1,J_2,g_1,g_2,k)$, set 
\[n_1=|J_1|+k \textit{ and } n_2=|J_2|+k.\] Then $(g_1,n_1,g_2,n_2)$ is a pair of non-negative integers contained in $\bigcup_{m=1}^{\lfloor \frac{\chi}{2}\rfloor}\mathcal{I}_m$, where $\chi=2g-2+n>1$ and $\mathcal{I}_m \ (1\leq m\leq \lfloor \frac{\chi}{2}\rfloor)$ is defined at the beginning of Subsection \ref{sub-I}.
Conversely, for any pair $(g_1,n_1,g_2,n_2)\in\bigcup_{m=1}^{\lfloor\frac{\chi}{2}\rfloor}\mathcal{I}_m$, one can randomly choose $|J_1|=n_1-k$ punctures in all $n$ punctures, so there exist at most $\binom{n}{n_1-k}$ different $\mathfrak{I}'s$ corresponding to it, where $1\leq k=\frac{n_1+n_2-n}{2}\leq n_1$.
\end{enumerate}
\end{remark}
Now we define two special separating simple closed multi-curves.
\begin{Def*}
\ben
\item[(i)] Assume $1\leq i<j\leq n$ and $\mathfrak{I}=(\{i,j\},\{1,2,...,n\}\setminus\{i,j\},0,g,1)$. We define $\Gamma_{ij}=\Gamma(\mathfrak{I})\subseteq S_{g,n}$ to be a simple closed curve bounding punctures $p_i$ and $p_j$, i.e., 
 $$S_{g,n}\setminus\Gamma(\mathfrak{I})\cong S_{0,3}\cup S_{g,n-1}$$
where the part $S_{0,3}$ contains punctures $p_i$ and $p_j$.

\item[(ii)] Assume $k\geq 1$ and $\{i_1,j_1, \cdots, i_k, j_k\}$ is a set of mutually distinct integers in $\{1,2,...,n\}$. Set
$$\mathfrak{P}=(i_1,j_1,...,i_k,j_k).$$
We define $\Gamma_{\mathfrak{P}}=\sum_{m=1}^k\gamma_m$ to be a simple closed multi-curve such that for all $1\leq m\leq k$, $\gamma_m\in\textnormal{Mod}_{g,n}\cdot\Gamma_{i_mj_m}$, i.e., 
$$S_{g.n}\setminus\Gamma_\mathfrak{P}\cong\left(\bigcup\limits_{m=1}^k S_{0,3}^m\right)\cup S_{g,n-k}$$
where each $S_{0,3}^m$, which is homeomorphic to $S_{0,3}$, contains punctures $p_{i_m}$ and $p_{j_m}$ for all $1\leq m \leq k$.
\een
\end{Def*}
(See Figure \ref{fig:00} for an illustration).
\begin{figure}[ht]
\centering
\includegraphics[width=0.90\textwidth]{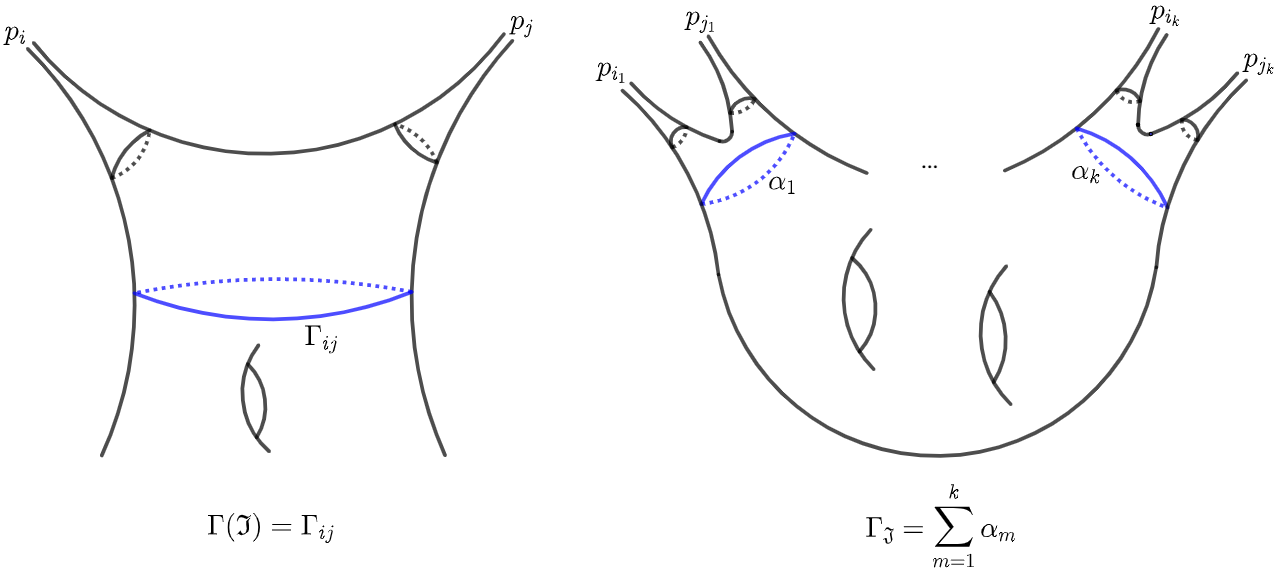}
\caption{Two types of separating simple closed multi-curves}
\label{fig:00}
\end{figure}

\begin{remark}\label{r-2}
For $\mathfrak{P}=(i_1,j_1,...,i_k,j_k)$ and $\mathfrak{P}'=(i_1',j_1',...,i_k',j_k')$, we say $\mathfrak{P}\sim\mathfrak{P}'$ if 
$$\{i_m,j_m\}=\left\{i_{m}',j_{m}'\right\}\text{ for all }1\leq m\leq k.$$
Then for fixed $k\geq 1$, the number of equivalent classes of pairs of form $(i_1,j_1,...,i_k,j_k)$ is equal to 
\[\binom{n}{2}\times \binom{n-2}{2} \times \cdots \times  \binom{n-2(k-1)}{2}=\frac{n!}{2^k\cdot(n-2k)!}.\]
\end{remark}
\subsection{Random variables}
Let $X\in\mathcal{T}_{g,n}$ be a hyperbolic surface. For a simple closed multi-curve $\Gamma=\sum\limits_{i=1}^k\gamma_i$ on $S_{g,n}$, its length $\ell_\Gamma(X)$ is defined as
$$\ell_\Gamma(X)=\sum\limits_{i=1}^k \ell_{\gamma_i}(X).$$
 Similar as \cite{Mirz13,NWX20}, for any $X\in\mathcal{M}_{g,n}$ and $L>0$, the counting functions $N_\Gamma(X,L)$ and $N_\Gamma^{\infty}(X,L)$ are defined as
$$N_\Gamma(X,L)\overset{\textnormal{def}}{=}\#\{\alpha\in\mathcal{O}_\Gamma;\ \ell_\alpha(X)\leq L\}$$
and
$$N_\Gamma^{\infty}(X,L)\overset{\textnormal{def}}{=}\#\left\{\sum\limits_{i=1}^k\gamma_i\in\mathcal{O}_\Gamma;\ \ell_{\gamma_i}(X)\leq L\text{ for all }1\leq i\leq k\right\}.$$
\noindent Now we define
$$\mathcal{N}_{(0,3)}^{(g,n-1)}(X,L) \overset{\textnormal{def}}{=} \left\{\begin{matrix}&\text{simple closed geodesic }\alpha \text{ in }X\text{ such that }\\ &X\setminus\alpha\cong S_{0,3}\cup S_{g,n-1}\text{ and }\ell_{\alpha}(X)\leq L\end{matrix}\right\}$$
and $$N_{(0,3)}^{(g,n-1)}(X,L) \overset{\textnormal{def}}{=} \#\mathcal{N}_{(0,3)}^{(g,n-1)}(X,L).$$
Using the notations in Subsection \ref{sec-prob}, for $k\geq 1$ we define
$$N_{(0,3),k}^{(g,n-k)}(X,L)\overset{\textnormal{def}}{=}\left(N_{(0,3)}^{(g,n-1)}(X,L)\right)_k.$$
Then $N_{(0,3),k}^{(g,n-1)}(X,L)$ is equal to the number of \emph{ordered} pair $(\alpha_1,...,\alpha_k)$, where the $\alpha_i's \ (1\leq i\leq k)$ are mutually distinct simple closed geodesics in $\mathcal{N}_{(0,3)}^{(g,n-1)}(X,L)$. Recall that the Collar Lemma (see \eg \cite[Chapter 4]{Buser-book}) tells that any two distinct simple closed geodesics in $X$ of lengths less than $2\arcsinh 1$ have no intersection. It follows that if $L<2\arcsinh 1$, then
\begin{align}\label{e-3}
N_{(0,3),k}^{(g,n-k)}(X,L)=\sum\limits_{\mathfrak{P}}N_{\Gamma_\mathfrak{P}}^\infty(X,L),
\end{align}
where $\mathfrak{P}$ runs over all different equivalent classes of pairs $(i_1,j_1,...,i_k,j_k)$ of mutually distinct integers in $\{1,2,...,n\}$. Then for any $k\geq 1$, $$N_{(0,3),k}^{(g,n-k)}(X,L):\mathcal{M}_{g,n}\to\mathbb{N}\cup\{0\}$$ is a random variable under the Weil-Petersson probability measure $\textnormal{Prob}_{\textnormal{WP}}^{g,n}$.

Now we prove two estimates on expected values which will be applied later.
 \begin{lemma}\label{l-exp}
For any fixed $0<L<2\arcsinh 1$ and $k\geq 1$, if $n\prec\sqrt g$, then for large enough $g$ we have
 $$\mathbb{E}_{\textnormal{WP}}^{g}\left[N_{(0,3),k}^{(g,n-k)}(X,L)\right]=\prod\limits_{i=0}^{2k-1}(n-i)\cdot\frac{V_{g,n-k}}{V_{g,n}}\cdot \left(2\cosh\frac{L}{2}-2\right)^k\cdot\left(1+O\left(\frac{nk^3L^2}{g}\right)\right),$$
 where the implied constant is independent of $g$ and $n$.
 \end{lemma}
 \begin{proof}
Define a function $F_k:\mathbb{R}^k_{\geq 0}\to\mathbb{R}_{\geq 0}$ as
$$F_k(x_1,...,x_k)=\prod_{i=1}^k 1_{[0,L]}(x_i).$$
From Proposition \ref{p-vol-1}, for any $0\leq x_1,...,\ x_k\leq L$,
 \begin{align}\label{l-exp-e}
 V_{g,n-k}(x_1,...,x_k,0,...,0)=V_{g,n-k}\cdot\prod\limits_{i=1}^k\frac{\sinh \left(\frac{x_i}{2}\right)}{\frac{x_i}{2}}\cdot \left(1+O\left(\frac{nk^3L^2}{g}\right)\right)
 \end{align}
 where the implied constant is independent of $g$ and $n$. For any $\mathfrak{P}=(i_1,j_1,...,i_k,j_k)$ where $\{i_1,j_1, \cdots, i_k, j_k\}$ is a set of mutually distinct integers in $\{1,2,...,n\}$. Let $\Gamma_{\mathfrak{P}}$ be its corresponding simple closed multi-curve defined above. Write $\Gamma_{\mathfrak{P}}=\sum_{t=1}^k\gamma_t$, where $\gamma_t(1\leq t\leq k)$ is a simple closed geodesic such that for each $t$,
$$S_{g,n}\setminus\gamma_t\cong S_{0,3}\cup S_{g,n-1}$$
 and the two punctures $p_{i_t}$ and $p_{j_t}$ are contained in the part of $S_{0,3}$. Recall that $M(\Gamma)=|\{i;\ \gamma_i\text{ separates off a one-handle from }S_{g,n}\}|$. So it is clear that \begin{align}\label{e-no11}M(\Gamma_{\mathfrak{P}})=0.\end{align} For any $f\in\textnormal{
Stab}(\Gamma_{\mathfrak{P}})$, since it fixes all punctures, it follows that for all $1\leq t\leq k$, $f$ fixes $\gamma_i$ and its orientation. Hence $\textnormal{Stab}(\Gamma_{\mathfrak{P}})=\cap_{t=1}^k\textnormal{Stab}^+(\gamma_t)$, which together with \eqref{e-coff} and \eqref{e-no11} implies that the coffecient $C_{\Gamma_{\mathfrak{P}}}$ in Theorem \ref{Mirz-for} satisfies that
\begin{align}\label{e-j-1}
C_{\Gamma_{\mathfrak{P}}}=\frac{2^{-M(\Gamma_{\mathfrak{P}})}}{|\textnormal{Stab}(\Gamma_{\mathfrak{P}}):\cap_{t=1}^k\textnormal{Stab}^+(\gamma_t)|}=1.
\end{align}
\noindent Now we apply Mirzakhani's Integration Formula, i.e., Theorem \ref{Mirz-for}, to  $F_k$ and $\Gamma_{\mathfrak{P}}$, together with \eqref{e-j-1} and the fact that $V_{0,3}(x,y,z)=1$ for all $x,y,z\geq 0$, by using \eqref{l-exp-e} one may conclude that
\begin{align*}
&\ \ \ \ \mathbb{E}_{\textnormal{WP}}^g\left[N_{\Gamma_{\mathfrak{P}}}^\infty(X,L)\right]\\
&=\frac{1}{V_{g,n}}\underbrace{\int_0^L...\int_0^L}_{k} V_{g,n-k}(x_1,...,x_k,0,..,0)\cdot \left(\prod_{i=1}^k V_{0,3}(0,0,x_i)x_i\right)dx_1...dx_k\\
&=\frac{1}{V_{g,n}}\underbrace{\int_0^L...\int_0^L}_{k} V_{g,n-k}\cdot \left(1+O\left(\frac{nk^3L^2}{g}\right)\right)\cdot \left(\prod_{i=1}^k 2\sinh \left(\frac{x_i}{2}\right)\right)dx_1...dx_k\\
&=\frac{V_{g,n-k}}{V_{g,n}}\cdot\left(\int_0^L 2\sinh\frac{x}{2}dx\right)^k\cdot \left(1+O\left(\frac{nk^3L^2}{g}\right)\right)\\
&=\frac{V_{g,n-k}}{V_{g,n}}\cdot \left(4\cosh\frac{L}{2}-4\right)^k\cdot \left(1+O\left(\frac{nk^3L^2}{g}\right)\right),
\end{align*}
where the implied constant is independent of $g$ and $n$. Together with the assumption that $L<2\arcsinh 1$ and \eqref{e-3}, we have
\begin{align*}
&\mathbb{E}_{\textnormal{WP}}^g\left[N_{(0,3),k}^{(g,n-k)}(X,L)\right]
=\sum\limits_{\mathfrak{P}}\mathbb{E}_{\textnormal{WP}}^g\left[N_{\Gamma_\mathfrak{P}}^{\infty}(X,L)\right]\\
&=\prod\limits_{i=0}^{2k-1} (n-i)\cdot\frac{V_{g,n-k}}{V_{g,n}}\cdot \left(2\cosh\frac{L}{2}-2\right)^k \cdot \left(1+O\left(\frac{nk^3L^2}{g}\right)\right),
\end{align*}
where in the second equation we apply the fact in Remark \ref{r-2} saying that there are $\frac{n!}{2^k\cdot(n-2k)!}$ different equivalent classes of pairs of form $\mathfrak{P}=(i_1,j_1,...,i_k,j_k)$ in $\{1,2, \cdots, n\}$. The proof is complete.
 \end{proof}

Without the assumption that $n\prec\sqrt g$ in Lemma \ref{l-exp}, we have 
 \begin{lemma}\label{l-exp-s}
For any fixed $k\geq 1$ and function $L:\mathbb{N} \cup\{0\}\times \mathbb{N} \cup\{0\} \to\mathbb{R}_{>0}$ with $$\lim\limits_{g+n\to\infty}L(g,n)=0,$$ then for large enough $(g+n)$ we have
 $$\mathbb{E}_{\textnormal{WP}}^{g,n}\left[N_{(0,3),k}^{(g,n-k)}(X,L(g,n))\right]=\prod\limits_{i=0}^{2k-1}(n-i)\cdot\frac{V_{g,n-k}}{V_{g,n}}\cdot \frac{L(g,n)^{2k}}{4^k}\cdot\left(1+O\left(kL(g,n)\right)\right)$$
 where the term $O(kL(g,n))$ is positive and the implied constant is independent of $g$ and $n$.
 \end{lemma}
 \begin{proof}
From Part $(1)$ of Lemma \ref{l-mirz}, for any $k\geq 1$ and $x_i\geq 0\ (1\leq i\leq k)$,
\be 
V_{g,n-k}\leq V_{g,n-k}(x_1,...,x_k,0,...,0)\leq e^{\frac{x_1+...+x_k}{2}}\cdot V_{g,n-k}.\nonumber
\ene
 It follows that if $x_i\leq L(g,n)$ for all $1\leq i\leq k$, then
 \begin{align}\label{e-3-11}
 V_{g,n-k}(x_1,...,x_k,0,...,0)=V_{g,n-k}\cdot(1+O(kL(g,n))),
 \end{align}
where the term $O(kL(g,n))$ is positive and the implied constant is independent of $g$ and $n$. Define a function $F_k:\mathbb{R}_{\geq 0}^k\to\mathbb{R}_{\geq 0}$ as 
$$F_k(x_1,...,x_k)=\prod\limits_{i=1}^k 1_{[0,L]}(x_i).$$
Assume $\mathfrak{P}=(i_1,j_1,...,i_k,j_k)$ and $\Gamma_{\mathfrak{P}}$ is the corresponding simple closed multi curve. Similar as the proof of Lemma \ref{l-exp} we apply Mirzakhani's Integration Formula, i.e., Theorem \ref{Mirz-for}, to  $F_k$ and $\Gamma_{\mathfrak{P}}$. By using \eqref{e-3-11} one may conclude that
 \begin{align*}
&\ \ \ \ \mathbb{E}_{\textnormal{WP}}^{g,n}\left[N_{\Gamma_{\mathfrak{P}}}^\infty(X,L(g,n))\right]\\
&=\frac{1}{V_{g,n}}\underbrace{\int_0^{L(g,n)}...\int_0^{L(g,n)}}_{k} V_{g,n-k}(x_1,...,x_k,0,..,0)\cdot \left(\prod_{i=1}^k V_{0,3}(0,0,x_i)x_i\right)dx_1...dx_k\\
&=\frac{1}{V_{g,n}}\underbrace{\int_0^{L(g,n)}...\int_0^{L(g,n)}}_{k} V_{g,n-k} \cdot \left(1+O\left(kL(g,n)\right)\right)\cdot\prod_{i=1}^k x_idx_1...dx_k\\
&=\frac{V_{g,n-k}}{V_{g,n}}\cdot\left(\int_0^{L(g,n)} xdx\right)^k \cdot \left(1+O(kL(g,n))\right)\\
&=\frac{V_{g,n-k}}{V_{g,n}}\cdot\frac{L(g,n)^{2k}}{2^k}\cdot (1+O(kL(g,n))).
\end{align*}
Recall that Remark \ref{r-2} says that there are $\frac{n!}{2^k\cdot(n-2k)!}$ different equivalent classes of pairs of form $\mathfrak{P}=(i_1,j_1,...,i_k,j_k)$. Then together with \eqref{e-3}, we obtain
\begin{align*}
&\mathbb{E}_{\textnormal{WP}}^{g,n}\left[N_{(0,3),k}^{(g,n-k)}(X,L(g,n))\right]=\sum\limits_{\mathfrak{P}}\mathbb{E}_{\textnormal{WP}}^{g,n}\left[N_{\Gamma_\mathfrak{P}}^{\infty}(X,L(g,n))\right]\\
&=\prod\limits_{i=0}^{2k-1} \left(n-i\right)\cdot\frac{V_{g,n-k}}{V_{g,n}}\cdot \frac{L(g,n)^{2k}}{4^k} \cdot \left(1+O\left(kL(g,n)\right)\right)
\end{align*}
as desired.
 \end{proof}
 \section{Proofs of Theorem \ref{mt-1}, \ref{mt-3} and \ref{mt-4}}\label{sec-proof}
 In this section, we study the behavior of the Cheeger constants of random surfaces in $\sM_{g,n(g)}$, and complete the proofs of Theorem \ref{mt-1}, \ref{mt-3} and \ref{mt-4}.

\subsection{The case that $n(g)$ grows slower than $\sqrt g$}\label{mt-p-1}
In this subsection, we closely follow the approach in \cite{Mirz13} of Mirzakhani to complete the proof of Theorem \ref{mt-4}.

The following proposition is essential to Theorem \ref{mt-4}.
\begin{proposition}\label{lh-1}
Assume $n(g)$ satisfies that $\lim\limits_{g\to\infty}\frac{n(g)}{\sqrt g}=0.$
Then for any constant $0<C<\frac{\log 2}{2\pi}$,
$$\lim\limits_{g\to\infty}\Prob\left(X\in\mathcal{M}_{g,n(g)};\ H(X)> C\right)=1$$
where $H(X)$ is the geodesic Cheeger constant of $X$.
\end{proposition}
\begin{proof}
Write $n(g)=\omega(g)\cdot\sqrt g$, where $\omega(g):\mathbb{N}\to\mathbb{R}_{\geq 0}$ is a function satisfying
\[
\lim\limits_{g\to\infty}\omega(g)=0.
\]
Assume $X\in\mathcal{M}_{g,n(g)}$ is a hyperbolic surface. Recall that from Section \ref{multi-curve}, for any pair $\mathfrak{I}=(J_1,J_2,g_1,g_2,k)$ satisfying
\begin{enumerate}[(a)]
\item $J_1\sqcup J_2=\{1,2,...,n(g)\}$;
\item $g_1+g_2+k-1=g$;
\item $1\leq2g_1-2+\left|J_1\right|+k\leq2g_2-2+\left|J_2\right|+k$,
\end{enumerate}
then it corresponds to a simple closed multi-curve $\Gamma(\mathfrak{I})=\sum_{i=1}^k\gamma_i$ such that $$X\setminus\Gamma(\mathfrak{I})\cong S_{g_1,|J_1|+k}\cup S_{g_2,|J_2|+k}.$$
On the other hand, for any $\alpha=\bigcup_{i=1}^k \alpha_i\in\mathcal{SG}(X)$ defined in Subsection \ref{gch}, there exists a pair $\mathfrak{I}$ such that $\sum_{i=1}^k\alpha_i\in\textnormal{Mod}_{g,n(g)}\cdot\Gamma(\mathfrak{I})$ .
For any $L>0$, the counting function is defined by
$$N_{\Gamma(\mathfrak{I})}(X,L)=\#\{\alpha\in\textnormal{Mod}_{g,n(g)}\cdot\Gamma(\mathfrak{I});\ \ell_\alpha(X)\leq L\}.$$
Denote by $$\chi(\mathfrak{I})=\min\limits_{i=1,2}\{2g_i-2+|J_i|+k\}.$$  Then $H(X)\leq C$ is equivalent to that there exists a pair $\mathfrak{I}=(J_1,J_2,g_1,g_2,k)$ such that $N_{\Gamma(\mathfrak{I})}(X,2\pi\chi(\mathfrak{I})C)\geq 1$, which in particular implies that
\begin{align}\label{lh1-e2}
&\ \ \ \
\Prob\left(X\in\mathcal{M}_{g,n(g)};\ H(X)\leq C\right)\\
&\leq \sum\limits_{\mathfrak{I}}\Prob\left(X\in\mathcal{M}_{g,n(g)};\ N_{\Gamma(\mathfrak{I})}(X,2\pi\chi(\mathfrak{I})C)\geq 1\right)\nonumber\\
&\leq\sum\limits_{\mathfrak{I}}\mathbb{E}_{\textnormal{WP}}^g\left[N_{\Gamma(\mathfrak{I})}(X,2\pi\chi(\mathfrak{I})C)\right],\nonumber
\end{align}
where $\mathfrak{I}$ runs over all possible pairs $(J_1,J_2,g_1,g_2,k)$. For any $k\geq 1$ and $L>0$, function $\phi_{k,L}:\mathbb{R}^k_{\geq 0}\to\mathbb{R}_{\geq 0}$ is defined by
$$\phi_{k,L}(x_1,...,x_k)=1_{[0,L]}(x_1+...+x_k).$$
Now we apply Mirzakhani's Integration Formula, i.e., Theorem \ref{Mirz-for}, to $\phi_{k,2\pi\chi(\mathfrak{I})C}$ and $\Gamma(\mathfrak{I})$. Assume $\Gamma(\mathfrak{I})=\sum_{i=1}^k\gamma_i$, where $\gamma_i(1\leq i\leq k)$ is a simple closed non-separating multi-curve in $S_{g,n(g)}$ and $G_k$ is the symmetry group of $\{1,2,...,k\}$. Recall that $M(\Gamma)=|\{i;\ \gamma_i\text{ separates off a one-handle from }S_{g,n}\}|$. So we have
\begin{align}\label{e-no1}
M(\Gamma(\mathfrak{I}))=0.
\end{align}
For any $f\in\textnormal{Stab}(\Gamma(\mathfrak{I}))$, since it fixes all punctures, it follows that $f$ fixes the orientation of $\gamma_i$ for all $1\leq i\leq k.$
 Define a homomorphism $\Psi:\textnormal{Stab}(\Gamma(\mathfrak{I}))\to G_k$ in the following way: for any $f\in\textnormal{Stab}(\Gamma(\mathfrak{I}))$ and $1\leq i\leq k$,
$$\Psi(f)(i)=j,\text{ where }1\leq j\leq k\text{ such that }f(\gamma_i)=\gamma_j.$$
It is not hard to check that $\Psi$ is surjective and $\textnormal{ker}(\Psi)=\cap_{i=1}^k\textnormal{Stab}^+(\gamma_i)$. Hence $$\textnormal{Stab}(\Gamma(\mathfrak{I}))/\cap_{i=1}^k\textnormal{Stab}^+(\gamma_i)\cong G_k,$$
which together with \eqref{e-coff} and \eqref{e-no1} implies that the coffecient $C_{\Gamma(\mathfrak{I})}$ in Theorem \ref{Mirz-for} satisfies that
\begin{align*}
C_{\Gamma(\mathfrak{I})}&=\frac{2^{-M(\Gamma(\mathfrak{I}))}}{|\textnormal{Stab}(\Gamma(\mathfrak{I})):\cap_{i=1}^k\textnormal{Stab}^+(\gamma_i)|}=\frac{1}{k!}.
\end{align*}
Thus it follows from Part $(1)$ of Lemma \ref{l-mirz} and  Theorem \ref{Mirz-for} that
\begin{align}\label{lh1-e4}
&\ \ \ \ \mathbb{E}_{\textnormal{WP}}^g\left[N_{\Gamma(\mathfrak{I})}(X,2\pi\chi(\mathfrak{I})C)\right]\\
&=\frac{1}{k!}\cdot \frac{1}{V_{g,n(g)}} \ \int\limits_{\sum_{i=1}^k x_i\leq 2\pi\chi(\mathfrak{I})C}V_{g_1,|J_1|+k}(x_1,...,x_k,0,...,0)\nonumber\\
&\times V_{g_2,|J_2|+k}(x_1,...,x_k,0,...,0)\times\prod_{i=1}^k x_idx_1...dx_k \nonumber\\
&\leq\frac{V_{g_1,|J_1|+k} \cdot V_{g_2,|J_2|+k}}{k!V_{g,n(g)}}\int\limits_{\sum_{i=1}^k x_i\leq 2\pi\chi(\mathfrak{I})C}e^{x_1+...+x_k}\prod_{i=1}^k x_idx_1...dx_k \nonumber\\
&\leq\frac{V_{g_1,|J_1|+k}\cdot V_{g_2,|J_2|+k}}{V_{g,n(g)}}\cdot\frac{(2\pi\chi(\mathfrak{I})C)^{2k}}{k!(2k)!}\cdot e^{2\pi\chi(\mathfrak{I})C},\nonumber
\end{align}
where in the last inequality we apply
$$\int\limits_{\sum_{i=1}^k x_i\leq 2\pi\chi(\mathfrak{I})C}\prod\limits_{i=1}^k x_idx_1...dx_k=\frac{(2\pi\chi(\mathfrak{I})C)^{2k}}{(2k)!}.$$
For $m\geq 1$, set $$L_m=2\pi mC+3(2\pi mC)^{\frac{2}{3}}.$$ Recall that from Part (2) of Remark \ref{r-1}, for any $\mathfrak{I}=(J_1,J_2,g_1,g_2,k)$ satisfying the conditions $(a)$, $(b)$ and $(c)$ above, it corresponds to a pair $(n_1,g_1,n_2,g_2)\in\mathcal{I}_m$ for some $1\leq m\leq\lfloor \frac{\chi}{2}\rfloor$. Conversely for any $(n_1,g_1,n_2,g_2)\in\mathcal{I}_m$, there are at most $\binom{n(g)}{n_1-k}$ different $\mathfrak{I}'s$ corresponding to it, where $1\leq k=\frac{n_1+n_2-n}{2}\leq n_1$. Then it follows that
\begin{align}\label{lh1-e8}
&\ \ \ \ \sum\limits_{\mathfrak{I}}\frac{V_{g_1,|J_1|+k}\cdot V_{g_2,|J_2|+k}}{V_{g,n(g)}}\cdot\frac{(2\pi\chi(\mathfrak{I})C)^{2k}}{k!(2k)!}\cdot e^{2\pi\chi(\mathfrak{I})C}\\
&\leq\sum\limits_{m=1}^{\left\lfloor \frac{\chi}{2}\right\rfloor}\sum\limits_{(g_1,n_1,g_2,n_2)\in\mathcal{I}_m}
\sum_{k=1}^{n_1}\binom{n(g)}{n_1-k}\cdot\frac{V_{g_1,n_1}\cdot V_{g_2,n_2}}{V_{g,n(g)}}\cdot
\frac{(2\pi mC)^{2k}}{k!(2k)!}\cdot e^{2\pi mC}\nonumber\\
&\leq\sum\limits_{m=1}^{\left\lfloor\frac{\chi}{2}\right\rfloor}\sum\limits_{(g_1,n_1,g_2,n_2)\in\mathcal{I}_m}
\sum_{k=1}^{n_1}C(m,n(g))\cdot\frac{V_{g_1,n_1}\cdot V_{g_2,n_2}}{V_{g,n(g)}}\cdot
\frac{(2\pi mC)^{2k}}{k!(2k)!}\cdot e^{2\pi mC}\nonumber
\end{align}
where $\mathfrak{I}$ runs over all possible pairs $(J_1,J_2,g_1,g_2,k)$ and $$C(m,n(g))=\max\limits_{0\leq i\leq m+1}\binom{n(g)}{i}.$$
It is clear that
\begin{align}\label{lh1-e5}
e^{3(2\pi mC)^{\frac{2}{3}}}=\left(1+\sum\limits_{k=1}^\infty \frac{(2\pi mC)^{\frac{2k}{3}}}{k!} \right)^3>1+\sum\limits_{k=1}^\infty \frac{(2\pi mC)^{2k}}{(k!)^3}>\sum\limits_{k=1}^\infty \frac{(2\pi mC)^{2k}}{k!(2k)!}.
\end{align}
Then combining \eqref{lh1-e2}, \eqref{lh1-e4}, \eqref{lh1-e8} and \eqref{lh1-e5}, we have
\begin{align}\label{lh1-e3}
&\ \ \ \ \Prob\left(X\in\mathcal{M}_{g,n(g)};\ H(X)\leq C\right)\\
&\leq\sum\limits_{m=1}^{\left\lfloor \frac{\chi}{2}\right\rfloor}\sum\limits_{(g_1,n_1,g_2,n_2)\in\mathcal{I}_m}
\sum_{k=1}^{n_1}C(m,n(g))\cdot\frac{V_{g_1,n_1}\cdot V_{g_2,n_2}}{V_{g,n(g)}}\cdot
\frac{(2\pi mC)^{2k}}{k!(2k)!}\cdot e^{2\pi mC}\nonumber\\
&=\frac{1}{V_{g,n(g)}}\sum\limits_{m=1}^{\left\lfloor \frac{\chi}{2}\right\rfloor}\sum\limits_{(g_1,n_1,g_2,n_2)\in\mathcal{I}_m}
C(m,n(g))V_{g_1,n_1}\cdot V_{g_2,n_2}e^{2\pi m C}\cdot\left(\sum\limits_{k=1}^{n_1}\frac{(2\pi mC)^{2k}}{k!(2k)!}\right)\nonumber\\
&\prec \frac{1}{V_{g,n(g)}}\sum\limits_{m=1}^{\left\lfloor\frac{\chi}{2}\right\rfloor}\sum\limits_{(g_1,n_1,g_2,n_2)\in\mathcal{I}_m}
C(m,n(g))V_{g_1,n_1}\cdot V_{g_2,n_2}e^{L_m}.\nonumber
\end{align}
Divide the summation  above into two disjoint parts $A_1$ and $A_2$, where $A_1$ is the summation of the terms when $m=1$ and $A_2$ is the remaining summation. From Proposition \ref{p-vol-2}, we have
\begin{align}\label{lh1-e6}
A_2\prec \frac{1}{\sqrt g}.
\end{align}
Now we consider the summation $A_1$. 
For any pair $(g_1,n_1,g_2,n_2)\in\mathcal{I}_1$, we have
\begin{align}\label{e-def}
2g_1-2+n_1=1\text{ and }2g_2-2+n_2=2g-2+n(g)-1.
\end{align}
The first equation in \eqref{e-def} tells that both $g_1$ and $n_1$ are uniformly bounded, hence
$$V_{g_1,n_1}\prec 1.$$
The second equation in \eqref{e-def} tells that
$$n_2-2(g-g_2)=n(g)-1,$$
which together with Lemma \ref{l-mirz} implies that
$$\frac{V_{g_2,n_2}}{V_{g,n(g)}}\leq\frac{V_{g,n(g)-1}}{V_{g,n(g)}}\prec\frac{1}{2g-2+n(g)}\asymp \frac{1}{g}.$$
Since $C(1,n(g))\asymp n(g)^2$ and the cardinality of pairs in $\mathcal{I}_1$ is uniformly bounded, it follows that
\begin{align}\label{lh1-e7}
A_1&=\frac{1}{V_{g,n(g)}}\sum\limits_{(g_1,n_1,g_2,n_2)\in\mathcal{I}_1}
C(1,n(g))V_{g_1,n_1}\cdot V_{g_2,n_2}\cdot e^{L_1}\\
&\prec\frac{n(g)^2}{g}=\omega(g)^2.\nonumber
\end{align}
Combining \eqref{lh1-e3}, \eqref{lh1-e6} and \eqref{lh1-e7}, we have
$$\Prob\left(X\in\mathcal{M}_{g,n(g)};\ H(X)\leq C\right)\prec \left(\omega(g)^2+\frac{1}{\sqrt g}\right).$$
Then the proof is complete by letting $g\to \infty$.
\end{proof}

Now we are ready to prove Theorem \ref{mt-4}.

\begin{proof}[Proof of Theorem \ref{mt-4}]
From Proposition \ref{lh-mirz} and \ref{lh-1}, we have for any $\epsilon>0$,
$$\lim\limits_{g\to\infty}\Prob\left(X\in\mathcal{M}_{g,n(g)};\ h(X)\geq \frac{\log 2}{\log 2+2\pi}-\epsilon\right)=1.$$
Then the proof follows from Cheeger's inequality, i.e., Lemma \ref{l-ch}.
\end{proof}
\subsection{The case that $n(g)$ is uniformly comparable to $\sqrt g$}
In this subsection, we consider the case that $n(g)\asymp\sqrt g$ and complete the proof of Theorem \ref{mt-3}.

\begin{figure}[ht]
\centering
\includegraphics[width=1.00\textwidth]{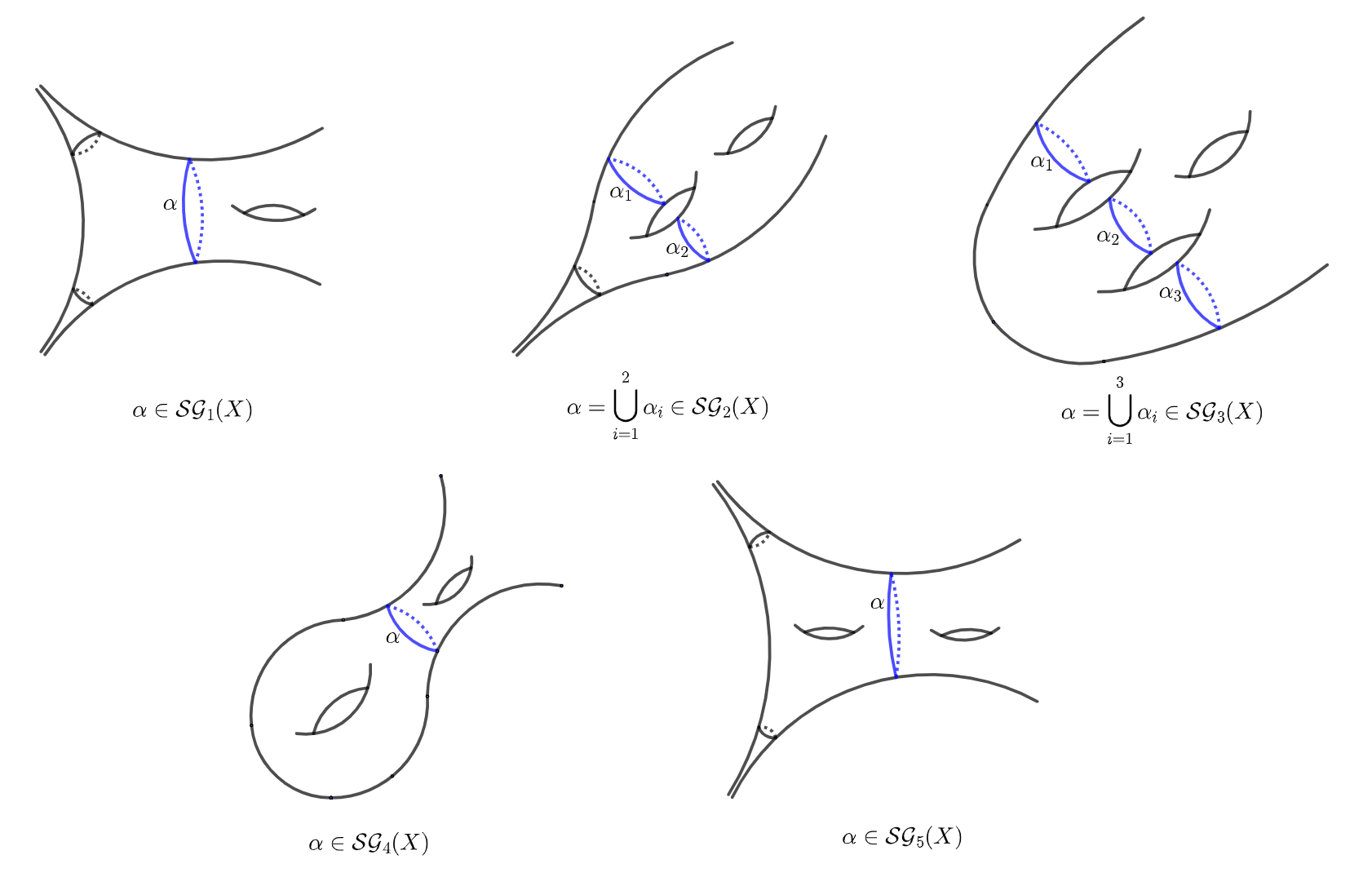}
\caption{A partition of $\mathcal{SG}(X)$}
\label{fig:01}
\end{figure}

For any hyperbolic surface $X\in\mathcal{M}_{g,n(g)}$, recall that $\mathcal{SG}(X)$ is a set of simple separating multi-geodesics, defined in Subsection \ref{gch}. Divide $\mathcal{SG}(X)$ into following five pairwisely disjoint parts (see Figure \ref{fig:01} for an illustration):
\begin{enumerate}
\item $\mathcal{SG}_{1}(X)=\left\{\alpha\in\mathcal{SG}(X);\ X\setminus\alpha\cong S_{0,3}\cup S_{g,n(g)-1}\right\}$.

\item $\mathcal{SG}_{2}(X)=\left\{\alpha=\bigcup_{i=1}^2\alpha_i\in\mathcal{SG}(X);\ X\setminus\alpha\cong S_{0,3}\cup S_{g-1,n(g)+1}\right\}$.
\item $\mathcal{SG}_{3}(X)=\left\{\alpha=\bigcup_{i=1}^3\alpha_i\in\mathcal{SG}(X);\ X\setminus\alpha\cong S_{0,3}\cup S_{g-2,n(g)+3}\right\}$.
\item $\mathcal{SG}_{4}(X)=\left\{\alpha\in\mathcal{SG}(X);\ X\setminus\alpha\cong S_{1,1}\cup S_{g-1,n(g)+1}\right\}$.
\item $\mathcal{SG}_5(X)=\mathcal{SG}(X)\setminus\bigcup_{i=1}^4\mathcal{SG}_i(X).$
\end{enumerate}

For each $1\leq i\leq 5$ and any $C>0$, the probability $P_{i,g}(C)$ is defined as
$$P_{i,g}(C)\overset{\textnormal{def}}{=}\Prob\left(X\in\mathcal{M}_{g,n(g)};\text{ there exists an $\alpha\in\mathcal{SG}_i(X)$ with }\mathcal{H}(\alpha)\leq C\right)$$ where $\mathcal{H}(\alpha)$ is defined in \eqref{def-h}. We also define the following two sets as
$$\mathcal{A}_{g,n(g)}(C)\overset{\textnormal{def}}{=}\left\{X\in\mathcal{M}_{g,n(g)};\ \begin{matrix}\text{there exists an }\alpha\in\bigcup_{j=2}^5\mathcal{SG}_j(X)\text{ with }  \mathcal{H}(\alpha)\leq C\end{matrix}\right\}$$
and
$$\mathcal{B}_{g,n(g)}(C)\overset{\textnormal{def}}{=}\left\{X\in\mathcal{M}_{g,n(g)};\ \begin{matrix}\text{there exists an }\alpha\in\mathcal{SG}_1(X) \text{ with }\mathcal{H}(\alpha)\leq C\end{matrix}\right\}.$$

Recall that the geometric Cheeger constant $H(X)$ of a hyperbolic surface $X$ is defined as 
$$H(X)=\inf\limits_{\alpha\in\mathcal{SG}(X)}\mathcal{H}(\alpha).$$
So we clearly have
\begin{align}\label{e-contain}
\mathcal{B}_{g,n(g)}(C)\subseteq \{X\in\mathcal{M}_{g,n(g)};\ H(X)\leq C\}\subseteq \mathcal{A}_{g,n(g)}(C)\cup\mathcal{B}_{g,n(g)}(C).
\end{align}
Moreover, we have the following theorem about $H(X)$.
\begin{theorem}\label{t-H}
Assume $n(g)$ satisfies that
$\lim\limits_{g\to\infty}\frac{n(g)}{\sqrt g}=a$
for some $a\in(0,\infty)$. Then for any $0<C<\frac{\log 2}{2\pi}$,
$$\lim\limits_{g\to\infty}\Prob\left(X\in\mathcal{M}_{g,n(g)};\ H(X)\leq C\right)=1-e^{-\lambda(a,C)},$$
where $\lambda(a,C)=\frac{a^2}{4\pi^2}\left(\cosh\pi C-1\right)$.
\end{theorem}
For the proof of Theorem \ref{t-H}, we firstly prove the following two lemmas on $\mathcal{A}_{g,n(g)}(C)$ and $\mathcal{B}_{g,n(g)}(C)$.

\begin{lemma}\label{l-h1}
Assume $n(g)\prec\sqrt g$ and $0<C<\frac{\log 2}{2\pi}$. Then we have
$$\lim\limits_{g\to\infty}\textnormal{Prob}_{\textnormal{WP}}^g\left(X\in\mathcal{M}_{g,n(g)};\ X\in\mathcal{A}_{g,n(g)}(C)\right)=0.$$
\end{lemma}
\begin{proof}
For $X\in\mathcal{M}_{g,n(g)}$, denote by $\{p_1,...,p_n\}$ the set of all punctures of $X$. From the definition of $\mathcal{A}_{g,n(g)}$, $X\in\mathcal{A}_{g,n(g)}$ is equivalent to that there exists an $\alpha\in\bigcup_{j=2}^5\mathcal{SG}_j(X)$ such that $\mathcal{H}(\alpha)\leq C$. It follows that
$$\textnormal{Prob}_{\textnormal{WP}}^g\left(X\in\mathcal{M}_{g,n(g)};\ X\in\mathcal{A}_{g,n(g)}(C)\right)\leq\sum\limits_{j=2}^5P_{j,g}(C).$$
It suffices to prove that for each $2\leq j\leq 5$,
$$P_{j,g}(C)\prec \frac{1}{\sqrt g}.$$
We first consider the case $j=2$. For all $1\leq i\leq n(g)$, assume $\Gamma_i=\gamma_{i1}+\gamma_{i2}$ is a simple closed multi-curve such that $\gamma_{i1}$ and $\gamma_{i2}$ bound the puncture $\{p_i\}$, i.e.,
$$X\setminus\Gamma_i\cong S_{0,3}\cup S_{g-1,n(g)+1},$$
where the puncture $\{p_i\}$ is contained in the part $S_{0,3}$. Then there exists an $\alpha\in\mathcal{SG}_i(X)$ with
$\mathcal{H}(\alpha)\leq C$ is equivalent to that there exists some $1\leq i\leq n(g)$ such that $N_{\Gamma_i}(X,2\pi C)\geq 1$. It follows that
\begin{align}\label{l-h1-e1}
P_{2,g}(C)&\leq\sum\limits_{i=1}^{n(g)}\Prob\left(X\in\mathcal{M}_{g,n(g)};\ N_{\Gamma_i}(X,2\pi C)\geq 1\right)\\
&\leq\sum\limits_{i=1}^{n(g)}\mathbb{E}_{\textnormal{WP}}^g\left[N_{\Gamma_i}\left(X,2\pi C\right)\right].\nonumber
\end{align}
Define a function $F:\mathbb{R}^2_{\geq 0}\to\mathbb{R}_{\geq 0}$ by
$$F(x,y)=1_{[0,2\pi C]}(x+y).$$
Now we apply Mirzakhani's Integration Formula, i.e., Theorem \ref{Mirz-for}, to $F$ and all simple closed multi-curves $\Gamma_i \ (1\leq i\leq n(g))$, together with Lemma \ref{l-mirz} we have
 \begin{align}\label{l-h2-e2}
 &\ \ \ \ \mathbb{E}_{\textnormal{WP}}^g\left[N_{\Gamma_i}\left(X,2\pi C\right)\right]\\
 &\leq
 \frac{1}{V_{g,n(g)}}\iint\limits_{\substack{x,y\geq 0\\x+y\leq 2\pi C}}V_{0,3}(x,y,0)V_{g-1,n(g)+1}(x,y,0,...,0)xydxdy\nonumber\\
 &\prec\frac{V_{g-1,n(g)+1}}{V_{g,n(g)}}\leq \frac{V_{g,n(g)-1}}{V_{g,n(g)}}\asymp \frac{1}{g+n(g)}.\nonumber
 \end{align}
 Since $n(g)\prec\sqrt g$, it follows from \eqref{l-h1-e1} and  \eqref{l-h2-e2} that
\begin{align*}
P_{2,g}(C)\prec \frac{n(g)}{g}\prec \frac{1}{\sqrt g}.
\end{align*}

Assume $\Gamma_{0,3}=\alpha_1+\alpha_2+\alpha_3$ is a simple closed multi-curve such that $$X\setminus\Gamma_{0,3}\cong S_{0,3}\cup S_{g-2,n(g)+3}$$
and
$\Gamma_{1,1}$ is a simple closed curve such that
$$X\setminus\Gamma_{1,1}\cong S_{1,1}\cup S_{g-1,n(g)+1}.$$
Similar as above, one may apply Mirzakhani's Integration Formula, i.e., Theorem \ref{Mirz-for}, to bound the expected values of $N_{\Gamma_{0,3}}(X,2\pi C)$ and $N_{\Gamma_{1,1}}(X,2\pi C)$ to obtain 
\begin{align*}
P_{3,g}(C)&\leq\mathbb{E}_{\textnormal{WP}}^g\left[N_{\Gamma_{0,3}}\left(X,2\pi C\right)\right]\\
&\leq
 \frac{1}{V_{g,n(g)}}\iiint\limits_{\substack{x,y,z\geq 0\\x+y+z\leq 2\pi C}}V_{0,3}(x,y,z)V_{g-2,n(g)+3}(x,y,z,0,...,0)xyzdxdydz\\
 &\prec\frac{V_{g-2,n(g)+3}}{V_{g,n(g)}}\leq \frac{V_{g,n(g)-1}}{V_{g,n(g)}}\asymp \frac{1}{g+n(g)}
\end{align*} 
and
\begin{align*}
P_{4,g}(C)&\leq\mathbb{E}_{\textnormal{WP}}^g\left[N_{\Gamma_{1,1}}\left(X,2\pi C\right)\right]\\
&\leq
 \frac{1}{V_{g,n(g)}}\int_0^{2\pi C}V_{1,1}(x)V_{g-1,n(g)+1}(x,0,...,0)xdx\\
 &\prec \frac{V_{g,n(g)-1}}{V_{g,n(g)}}\asymp \frac{1}{g+n(g)}.
\end{align*}

It remains to consider $P_{5,g}(C)$. The set $\mathcal{SG}_5(X)$ consists of all simple closed multi-geodesics $\Gamma$ in $X$ such that 
\begin{align}\label{e-geq2}
X\setminus\Gamma\cong X_1\cup X_2,
\end{align}
where $X_1,\ X_2$ are connected subsurfaces of $X$ with $|\chi(X_2)|\geq|\chi(X_1)|\geq 2$. For this case, based on our assumption that $0<C<\frac{\log 2}{2\pi}$, by following exactly the same proof of Equation \eqref{lh1-e6} one may deduce that 
\begin{align*}
P_{5,g}(C)\leq\sum\limits_{\Gamma}\mathbb{E}_{\textnormal{WP}}^g\left[N_{\Gamma}\left(X,2\pi C|\chi(X_1)|\right)\right]\leq A_2\prec \frac{1}{\sqrt g},
\end{align*}
where $\Gamma$ runs over all simple closed multi-geodesics satisfying \eqref{e-geq2}.
The proof is complete.
\end{proof}

\begin{lemma}\label{l-h2}
Assume $n(g)$ satisfies that $\lim\limits_{g\to\infty}\frac{n(g)}{\sqrt g}=a$ for some $a\in(0,\infty)$. Then for any $0<C<\frac{\arcsinh 1}{\pi}$,
$$\lim\limits_{g\to\infty}\textnormal{Prob}_{\textnormal{WP}}^g
\left(X\in\mathcal{M}_{g,n(g)};\ X\in\mathcal{B}_{g,n(g)}(C)\right)=1-e^{\lambda(a,C)},$$
where $\lambda(a,C)=\frac{a^2}{4\pi^2}(\cosh\pi C-1)$.
\end{lemma}
\begin{proof}
 From the definition of $\mathcal{B}_{g,n(g)}(C)$, we have
 $$P_{1,g}=\textnormal{Prob}_{\textnormal{WP}}^g
\left(X\in\mathcal{M}_{g,n(g)};\ X\in\mathcal{B}_{g,n(g)}(C)\right).$$
 It suffices to show that for any $0<C<\frac{\arcsinh 1}{\pi}$,
\begin{align*}
\lim\limits_{g\to\infty}P_{1,g}(C)=1-e^{-\lambda(a,C)}.
\end{align*}
Since $n(g)\asymp \sqrt{g}$, it follows from Lemma \ref{l-MZ} that for any fixed $k\in\mathbb{N}$, as $g\to\infty$
\begin{align}\label{t-H-e1}
\frac{V_{g,n(g)-k}}{V_{g,n(g)}}=\prod\limits_{i=1}^k\frac{V_{g,n(g)-i}}{V_{g,n(g)-i+1}}\sim\left(\frac{1}{4\pi^2}\right)^k\cdot\frac{1}{(2g)^k}.
\end{align}
 Since $2\pi C<2\arcsinh 1$, from \eqref{t-H-e1} and Lemma \ref{l-exp},
\begin{align*}
&\lim\limits_{g\to\infty}\mathbb{E}_{\textnormal{WP}}^g\left[\left(N_{(0,3)}^{(g,n(g)-1)}(X,2\pi C)\right)_k\right]
=\lim\limits_{g\to\infty}\mathbb{E}_{\textnormal{WP}}^g\left[N_{(0,3),k}^{(g,n(g)-k)}(X,2\pi C)\right]\\
&=\lim\limits_{g\to\infty}\prod\limits_{i=0}^{2k-1} (n(g)-i)\cdot\frac{V_{g,n(g)-k}}{V_{g,n(g)}}\cdot \left(2\cosh\frac{2\pi C}{2}-2\right)^k\\
&=\lim\limits_{g\to\infty}n(g)^{2k}\cdot\left(\frac{1}{4\pi^2}\right)^k\cdot\frac{1}{(2g)^k}\cdot\left(2\cosh\pi C-2\right)^k=\lambda(a,C)^k
\end{align*}
where $\lambda(a,C)=\frac{a^2}{4\pi^2}\left(\cosh\pi C-1\right)$. Then we apply Theorem \ref{t-pos} to the random variable $N_{(0,3)}^{(g,n(g)-1)}(X,2\pi C)$ to conclude that
\begin{align*}
&\lim\limits_{g\to\infty}P_{1,g}(C)=\lim\limits_{g\to\infty}\Prob\left(X\in\mathcal{M}_{g,n(g)};\ N_{(0,3)}^{(g,n(g)-1)}(X,2\pi C)\geq 1\right)\\
&=1-\lim\limits_{g\to\infty}\Prob\left(X\in\mathcal{M}_{g,n(g)};\ N_{(0,3)}^{(g,n(g)-1)}(X,2\pi C)=0\right)=1-e^{-\lambda(a,C)}.
\end{align*}
The proof is complete.
\end{proof}

Now we finish the proof of Theorem \ref{t-H} as follows.

\begin{proof}[Proof of Theorem \ref{t-H}]
From \eqref{e-contain}, we have
$$\textnormal{Prob}_{\textnormal{WP}}^g
\left(X\in\mathcal{M}_{g,n(g)};\ X\in\mathcal{B}_{g,n(g)}(C)\right)\leq\textnormal{Prob}_{\textnormal{WP}}^g
\left(X\in\mathcal{M}_{g,n(g)};\ H(X)\leq C\right)$$
and 
\begin{align*}
\textnormal{Prob}_{\textnormal{WP}}^g
\left(X\in\mathcal{M}_{g,n(g)};\ H(X)\leq C\right)&\leq\textnormal{Prob}_{\textnormal{WP}}^g
\left(X\in\mathcal{M}_{g,n(g)};\ X\in\mathcal{A}_{g,n(g)}(C)\right)\\
&+\textnormal{Prob}_{\textnormal{WP}}^g
\left(X\in\mathcal{M}_{g,n(g)};\ X\in\mathcal{B}_{g,n(g)}(C)\right).
\end{align*}
Since $\log 2<2\arcsinh 1$, the conclusion clearly follows from Lemma \ref{l-h1} and \ref{l-h2}.
\end{proof}
Before the proof of Theorem \ref{mt-3}, we prove the following lemma.
\begin{lemma}\label{l-describe}
Assume $0<C<\frac{\log 2}{\sqrt{4\pi(\log 2+\pi)}}$ and $\epsilon>0$ such that
\begin{align}\label{mt3-e-0}
\frac{\log 2-2\pi\epsilon}{2\pi(1-\epsilon)+\log 2}>\frac{C}{\sqrt{1+C^2}}.
\end{align}
Then we have
\begin{align*}
&\ \ \ \ \left\{X\in\mathcal{M}_{g,n(g)};\ h(X)\leq\frac{C}{\sqrt{1+C^2}}\right\}\\
&\subseteq\mathcal{A}_{g,n(g)}\left(\frac{\log 2}{2\pi}-\epsilon\right)\cup\left\{X\in\mathcal{M}_{g,n(g)};\ H(X)\leq C\right\}.\nonumber
\end{align*}
\end{lemma}
\begin{proof}
It suffices to prove that if $X\in\mathcal{M}_{g,n(g)}$ satisfies $$h(X)\leq\frac{C}{\sqrt{1+C^2}}\text{ and } X\notin\mathcal{A}_{g,n(g)}\left(\frac{\log 2}{2\pi}-\epsilon\right),$$ then $H(X)\leq C$. As in Subsection \ref{gch} we know that there exists an $\alpha\in\mathcal{S}(X)$ such that
\begin{align}\label{mt3-e-1}
\mathcal{H}(\alpha)\leq\frac{C}{\sqrt{1+C^2}}.
\end{align}
Moreover, the curve $\alpha$ must be one of following two types, i.e., $(1)$ and $(2)$ as in Theorem \ref{t-AM}:
\begin{enumerate}
\item horocycles around cusps;
\item geodesics or single ``neighboring curves" at constant distance from a geodesic.
\end{enumerate}

 \underline{Case I: $\alpha$ is of type-(1).} For this case, it is not hard to see that
   \begin{align*}
   \mathcal{H}(\alpha)\geq 1,
    \end{align*}
    which contradicts \eqref{mt3-e-1}.

  \underline{Case II: $\alpha$ is of type-(2).} Assume $\alpha=\bigcup_{i=1}^k\alpha_i$, then there exists an $\alpha'=\bigcup_{i=1}^k\alpha_i^\prime\in\mathcal{SG}(X)$ such that $\alpha_i$ is a neighboring curve at constant distance from geodesic $\alpha_i^\prime$ for all $1\leq i\leq k$. If $\alpha'\in\bigcup_{i=2}^5\mathcal{SG}_i(X)$, then from our assumption that $X\notin\mathcal{A}_{g,n(g)}\left(\frac{\log 2}{2\pi}-\epsilon\right)$, we have
  $$\mathcal{H}(\alpha')>\frac{\log 2}{2\pi}-\epsilon.$$
Together with \eqref{dh-3} and \eqref{mt3-e-0}, it follows that
\begin{align*}
\mathcal{H}(\alpha)&\geq\frac{\mathcal{H}(\alpha')}{1+\mathcal{H}(\alpha')}\geq\frac{\log 2-2\pi\epsilon}{2\pi(1-\epsilon)+\log 2}>\frac{C}{\sqrt{1+C^2}},
\end{align*}
which contradicts \eqref{mt3-e-1}.
So $$\alpha^\prime\in\mathcal{SG}_1(X).$$ In particular, we have that $\alpha^\prime$ is a simple closed geodesic such that $$X\setminus\alpha'\cong X\setminus\alpha\cong S_{0,3}\cup S_{g,n(g)-1}.$$ The classical Collar Lemma (see \eg \cite[Theorem 4.1.1]{Buser-book}) tells that $\alpha'$ admits a standard collar $\mathbb{S}^1\times [-w,w]$ endowed with the hyperbolic metric $ds^2=d\rho^2+\ell^2(\alpha')\cosh^2 \rho d\theta^2$ where $\rho$ is the hyperbolic distance to the central closed geodesic $\alpha'$ and $w=\arcsinh\left(\frac{1}{\sinh \left(\frac{\ell(\alpha')}{2} \right)} \right)$. Assume $X\setminus\alpha=X_1\cup X_2$ and $\textnormal{dist}(\alpha,\alpha')=t$ for some $t\geq 0$. Then we have $$\ell(\alpha)=\ell(\alpha')\cosh t\text{ and } \min\{\textnormal{Area}(X_1),\textnormal{Area}(X_2)\}=2\pi\pm\ell(\alpha')\sinh t.$$
Since $\ell(\alpha^\prime)\geq 2\pi H(X)$, it follows that
\begin{align}\label{mt3-e-2}
\mathcal{H}(\alpha)&=\frac{\ell(\alpha)}{\min\{\textnormal{Area}(X_1),\textnormal{Area}(X_2)\}}\\
&\geq\frac{\ell(\alpha')\cosh t}{2\pi+\ell(\alpha')\sinh t}\geq\frac{H(X)\cosh t}{1+H(X)\sinh t}.\nonumber
\end{align}
Denote function $\phi(s)=\frac{H(X)\cosh s}{1+H(X)\sinh s}$. A simple computation shows that
$\phi'(s)=\frac{H(X)(\sinh s-H(X))}{\left(1+H(X)\sinh s\right)^2}$ which in particular implies that
\begin{align}\label{mt3-e-3}
\phi(t)\geq\phi(\arcsinh (H(X)))=\frac{H(X)}{\sqrt{1+H(X)^2}}.
\end{align}
From \eqref{mt3-e-1}, \eqref{mt3-e-2} and \eqref{mt3-e-3}, we have
\begin{align*}
\frac{H(X)}{\sqrt{1+H(X)^2}}\leq\mathcal{H}(\alpha)\leq\frac{C}{\sqrt{1+C^2}},
\end{align*}
which implies that $H(X)\leq C$. The proof is complete.
\end{proof}
Now we turn to prove Theorem \ref{mt-3}.
\begin{proof}[Proof of Theorem \ref{mt-3}]
Since $0<C<\frac{\log 2}{\sqrt{4\pi(\log 2+\pi)}}$, one may take $\epsilon>0$ such that
$$\frac{\log 2-2\pi\epsilon}{2\pi(1-\epsilon)+\log 2}>\frac{C}{\sqrt{1+C^2}}.$$
First it follows from Lemma \ref{l-h1}, Lemma \ref{l-describe} and Theorem \ref{t-H} that
\begin{align}\label{mt3-e-4}
&\ \ \ \ \limsup\limits_{g\to\infty}\Prob\left(X\in\mathcal{M}_{g,n(g)};\ h(X)\leq\frac{C}{\sqrt{1+C^2}}\right)\\
&\leq\limsup\limits_{g\to\infty}\Prob\left(X\in\mathcal{M}_{g,n(g)};\ X\in\mathcal{A}_{g,n(g)}\left(\frac{\log 2}{2\pi}-\epsilon\right)\right)\nonumber\\
&+\limsup\limits_{g\to\infty}\Prob\left(X\in\mathcal{M}_{g,n(g)};\ H(X)\leq C\right)\nonumber\\
&=1-e^{-\lambda(a,C)}.\nonumber
\end{align}
For the other direction, if $X\in\mathcal{B}_{g,n(g)}(C)$, then there exists a $\beta'\in\mathcal{SG}_1(X)$ such that $\mathcal{H}(\beta')=\frac{\ell(\beta')}{2\pi}\leq C$. Assume $X\setminus\beta'=X_1\cup X_2$, where $X_2\cong S_{g,n(g)-1}$.
Consider the curve $\beta\subseteq X_2$ which is a neighboring curve at constant distance $\arcsinh (\mathcal{H}(\beta'))$ from $\beta'$. Since $\mathcal{H}(\beta')\leq C$, $\ell(\beta')\leq 2\pi C<\log 2$. So we have 
$$\sinh\left(\frac{\ell(\beta^\prime)}{2}\right)<\sinh \frac{\log 2}{2}=\frac{1}{2}\left(\sqrt 2-\frac{1}{\sqrt{2}}\right)<1$$
implying that
\[\mathcal{H}(\beta')\leq  C<\frac{\log 2}{2\pi}<1<\frac{1}{\sinh \left(\frac{\ell(\beta')}{2} \right)}.\]
This together with the classical Collar Lemma (see \eg  \cite[Theorem 4.1.1]{Buser-book}) implies that $\beta$ is a simple closed curve contained in the standard collar of $\beta'$. Moreover, $$h(X)\leq\mathcal{H}(\beta)=\frac{\ell(\beta')\cdot\cosh(\arcsinh (\mathcal{H}(\beta')))}{2\pi+\ell(\beta')\cdot\sinh(\arcsinh (\mathcal{H}(\beta')))}=\frac{\mathcal{H}(\beta')}{\sqrt{1+\mathcal{H}(\beta')^2}}\leq\frac{C}{\sqrt{1+C^2}}.$$
Then we have
$$\mathcal{B}_{g,n(g)}(C)\subseteq \left\{X\in\mathcal{M}_{g,n(g)};\ h(X)\leq\frac{C}{\sqrt{1+C^2}}\right\},$$
which together with Lemma \ref{l-h2} implies that
\begin{align}\label{mt3-e-5}
&\ \ \ \ \liminf\limits_{g\to\infty}\Prob\left(X\in\mathcal{M}_{g,n(g)};\ h(X)\leq\frac{C}{\sqrt{1+C^2}}\right)\\
&\geq \liminf\limits_{g\to\infty}\Prob\left(X\in\mathcal{M}_{g,n(g)};\ X\in\mathcal{B}_{g,n(g)}\right)\nonumber\\
&=1-e^{-\lambda(a,C)}.\nonumber
\end{align}
Then the conclusion follows from \eqref{mt3-e-4} and \eqref{mt3-e-5}.
\end{proof}
\subsection{The case that $n(g)$ grows faster than $\sqrt g$ and slower than $g$}
In this subsection we finish the proof of Theorem \ref{mt-1}. Recall that for any $X\in\mathcal{M}_{g,n(g)}$ and $L>0$,
$$\mathcal{N}_{(0,3)}^{(g,n(g)-1)}(X,L)=\left\{\begin{matrix}&\text{simple closed geodesic }\alpha \text{ in }X\text{ such that }\\ &X\setminus\alpha\cong S_{0,3}\cup S_{g,n(g)-1}\text{ and }\ell_{\alpha}(X)\leq L\end{matrix}\right\}$$
and
$$N_{(0,3)}^{(g,n(g)-1)}(X,L)=\#\mathcal{N}_{(0,3)}^{(g,n(g)-1)}(X,L).$$
Define a function $L:\mathbb{N}\to\mathbb{R}$ as $$L(g)=\left(\frac{\sqrt{g}}{n(g)}\right)^{\frac{1}{2}}\text{ for all }g\in\mathbb{N}.$$

The following proposition is essential to Theorem \ref{mt-1}.
\begin{proposition}\label{ph-1}
Assume $n(g)$ satisfies $\lim\limits_{g\to\infty}\frac{n(g)}{\sqrt g}=\infty$. Then we have
$$\liminf\limits_{g\to\infty}\Prob\left(X\in\mathcal{M}_{g,n(g)};\ N_{(0,3)}^{(g,n(g)-1)}(X,L(g))\geq 1\right)\geq\liminf\limits_{g\to\infty}\frac{V_{g,n(g)-1}^2}{V_{g,n(g)}V_{g,n(g)-2}}.$$
\end{proposition}
\begin{proof}
Notice that $\left(N_{(0,3)}^{(g,n(g)-1)}(X,L(g))\right)^2$ is the number of pairs $(\alpha,\beta)$, where $\alpha,\beta\in\mathcal{N}_{(0,3)}^{(g,n-1)}(X,L(g))$.
From the Collar Lemma (see \eg \cite[Theorem 4.1.6]{Buser-book}), if $L<2\arcsinh 1$, then for any such pair $(\alpha,\beta)$, either $\alpha=\beta$ or $\alpha$ and $\beta$ have no intersection. Hence for $g>0$ large enough we have
\begin{align}\label{exp-e1}
&\ \ \ \ \mathbb{E}_{\textnormal{WP}}^g\left[\left(N_{(0,3)}^{(g,n(g)-1)}(X,L(g))\right)^2\right]\\
&=\mathbb{E}_{\textnormal{WP}}^g\left[N_{(0,3)}^{(g,n(g)-1)}(X,L(g))\right]
+\mathbb{E}_{\textnormal{WP}}^g\left[N_{(0,3),2}^{(g,n(g)-2)}(X,L(g))\right].\nonumber
\end{align}
It follows from Lemma \ref{l-exp-s} that
\begin{align}\label{exp-e2-1}
\mathbb{E}_{\textnormal{WP}}^g\left[N_{(0,3)}^{(g,n(g)-1)}(X,L(g))\right]&\sim n(g)(n(g)-1)\cdot\frac{V_{g,n(g)-1}}{V_{g,n(g)}}\cdot \frac{L(g)^2}{4}\\
&\sim \frac{V_{g,n(g)-1}}{V_{g,n(g)}}\cdot\frac{n(g)^2L(g)^2}{4}.\nonumber
\end{align}
This together with Part $(2)$ of Lemma \ref{l-mirz} implies that
\begin{align}\label{exp-e2}
\mathbb{E}_{\textnormal{WP}}^g\left[N_{(0,3)}^{(g,n(g)-1)}(X,L(g))\right]\asymp\frac{\sqrt{g}\cdot n(g)}{g+n(g)}.
\end{align}
Similarly, it follows from Lemma \ref{l-exp-s} and Part $(2)$ of Lemma \ref{l-mirz} that
\begin{align}\label{exp-e3-1}
\mathbb{E}_{\textnormal{WP}}^g\left[N_{(0,3),2}^{(g,n(g)-2)}(X,L(g))\right]
&\sim\prod\limits_{i=0}^3(n(g)-i)\cdot\frac{V_{g,n(g)-2}}{V_{g,n}}\cdot\frac{L(g)^4}{16}\\
&\sim\frac{V_{g,n(g)-2}}{V_{g,n(g)}}\cdot\frac{n(g)^4L(g)^4}{16}\nonumber
\end{align}
and
\begin{align}\label{exp-e3}
\mathbb{E}_{\textnormal{WP}}^g\left[N_{(0,3),2}^{(g,n(g)-2)}(X,L(g))\right]\asymp \left(\frac{\sqrt{g}\cdot n(g)}{g+n(g)}\right)^2.
\end{align}
By our assumption that $\lim\limits_{g\to\infty}\frac{n(g)}{\sqrt g}=\infty$, it is clear that $\lim\limits_{g\to \infty}\frac{\sqrt{g}\cdot n(g)}{g+n(g)}=\infty$. Then it follows from \eqref{exp-e1}, \eqref{exp-e2} and \eqref{exp-e3} that
$$\mathbb{E}_{\textnormal{WP}}^g\left[\left(N_{(0,3)}^{(g,n(g)-1)}(X,L(g))\right)^2\right]\asymp\left(\frac{\sqrt{g}\cdot n(g)}{g+n(g)}\right)^2$$
and
\begin{align}\label{exp-e4}
&\ \ \ \ \left|\frac{\mathbb{E}_{\textnormal{WP}}^g\left[N_{(0,3)}^{(g,n(g)-1)}(X,L(g))\right]^2}{\mathbb{E}_{\textnormal{WP}}^g\left[\left(N_{(0,3)}^{(g,n(g)-1)}(X,L(g))\right)^2\right]
}-\frac{\mathbb{E}_{\textnormal{WP}}^g\left[N_{(0,3)}^{(g,n(g)-1)}(X,L(g))\right]^2}{\mathbb{E}_{\textnormal{WP}}^g\left[N_{(0,3),2}^{(g,n(g)-2)}(X,L(g))\right]}\right|\\
&=\frac{\mathbb{E}_{\textnormal{WP}}^g\left[N_{(0,3)}^{(g,n(g)-1)}(X,L(g))\right]^3}{\mathbb{E}_{\textnormal{WP}}^g\left[\left(N_{(0,3)}^{(g,n(g)-1)}(X,L(g))\right)^2\right]\cdot\mathbb{E}_{\textnormal{WP}}^g\left[N_{(0,3),2}^{(g,n(g)-2)}(X,L(g))\right]}\nonumber\\
&\asymp \frac{g+n(g)}{\sqrt{g}\cdot n(g)}.\nonumber
\end{align}
Now we apply Lemma \ref{l-cs} to random variable $N_{(0,3)}^{g,n(g)-1}(X,L(g))$. This together with \eqref{exp-e4} implies that
\begin{align}\label{e-pf3}
&\ \ \ \Prob\left(X\in\mathcal{M}_{g,n(g)};N_{(0,3)}^{(g,n(g)-1)}(X,L(g))\geq 1\right)\\
&\geq \frac{\mathbb{E}_{\textnormal{WP}}^g\left[N_{(0,3)}^{(g,n(g)-1)}(X,L(g))\right]^2}{\mathbb{E}_{\textnormal{WP}}^g\left[\left(N_{(0,3)}^{(g,n(g)-1)}(X,L(g))\right)^2\right]}\nonumber\\
&=\frac{\mathbb{E}_{\textnormal{WP}}^g\left[N_{(0,3)}^{(g,n(g)-1)}(X,L(g))\right]^2}{\mathbb{E}_{\textnormal{WP}}^g\left[N_{(0,3),2}^{(g,n(g)-2)}(X,L(g))\right]}-O\left( \frac{g+n(g)}{\sqrt{g}\cdot n(g)}\right),\nonumber
\end{align}
where the implied constant is independent of $g$. From  \eqref{exp-e2-1} and \eqref{exp-e3-1} we have
\begin{align}\label{e-pf-4-1}
&\frac{\mathbb{E}_{\textnormal{WP}}^g\left[N_{(0,3)}^{(g,n(g)-1)}(X,L(g))\right]^2}{\mathbb{E}_{\textnormal{WP}}^g\left[N_{(0,3),2}^{(g,n(g)-2)}(X,L(g))\right]}\\
&\sim\left(\frac{V_{g,n(g)-1}}{V_{g,n(g)}}\cdot\frac{n(g)^2L(g)^2}{4}\right)^2\cdot\frac{V_{g,n(g)}}{V_{g,n(g)-2}}\cdot\frac{16}{n(g)^4L(g)^4}\nonumber\\
&=\frac{V_{g,n(g)-1}^2}{V_{g,n(g)}V_{g,n(g)-2}}.\nonumber
\end{align}
By taking infimum limits, the conclusion then follows from \eqref{e-pf3} and \eqref{e-pf-4-1}.

\end{proof}

Now we are ready to prove Theorem \ref{mt-1}.
\begin{proof}[Proof of Theorem \ref{mt-1}]
Since $n(g)=o(g)$,  Lemma \ref{l-cor-3} tells that
\[\lim\limits_{g\to\infty}\frac{V_{g,n(g)-1}^2}{V_{g,n(g)}V_{g,n(g)-2}}=1.\]
This together with Proposition \ref{ph-1} implies that
\begin{align*}
\lim\limits_{g\to\infty}\Prob\left(X\in\mathcal{M}_{g,n(g)};\ N_{(0,3)}^{(g,n(g)-1)}(X,L(g))\geq 1\right)=1.
\end{align*}

\noindent Let $X\in \sM_{g,n(g)}$ with $N_{(0,3)}^{(g,n(g)-1)}(X,L(g))\geq 1$. Then the Cheeger constant satisfies 
\[h(X)\leq \frac{L(g)}{2\pi}.\]
Hence we have that for any $\epsilon>0$
$$\lim\limits_{g\to\infty}\Prob\left(X\in\mathcal{M}_{g,n(g)};\ h(X)\leq\epsilon\right)=1.$$
Then the conclusion follows by applying Theorem \ref{exist} and \ref{t-Bu}.
\end{proof}

\begin{rem*}
Recall that $L(g)=\left(\frac{\sqrt{g}}{n(g)}\right)^{\frac{1}{2}}$. Indeed, the upper bound $\epsilon$ in Theorem \ref{mt-1} can be replaced by $ \left(\frac{\sqrt{g}}{n(g)}\right)^{\frac{1}{2}}$ up to some uniform constant multiplication, that tends to $0$ as $g\to \infty$.
\end{rem*}

\section{One open question on WP volume}\label{sec-ques}
It is worthwhile to know whether the condition that $\lim\limits_{g\to\infty}\frac{n(g)}{g}=0$ in Theorem \ref{mt-1} can be removed. Based on Proposition \ref{ph-1}, the proof of Theorem \ref{mt-3} tells that it suffices to solve the following question:
\begin{que*}[=Question \ref{que-wpv}]
Does the following limit hold:
$$\lim\limits_{g+n\to\infty}\frac{V_{g,n}^2}{V_{g,n-1}V_{g,n+1}}=1$$ where $n$ may depend on $g$.
\end{que*}

In this section we prove the following result which partially answers Question \ref{que-wpv}.
\begin{proposition}\label{p-est}Let $n>0$ which may depend on the genus $g$. Then
\[\frac{1}{2}-\frac{\pi^2}{20}\leq \liminf\limits_{g+n\to\infty}\frac{V_{g,n}^2}{V_{g,n-1}V_{g,n+1}}\leq\limsup\limits_{g+n\to\infty}\frac{V_{g,n}^2}{V_{g,n-1}V_{g,n+1}}\leq 1.\]
\end{proposition}
\begin{proof}
We first show the LHS inequality. Recall that it is known by the Recursive Formula (see \eg \cite[Equation II on Page 276]{Mirz13}) that
\begin{align*}
\frac{(2g-2+n)V_{g,n}}{V_{g,n+1}}=\frac{1}{2}\sum\limits_{m=1}^{3g-2+n}(-1)^{m-1}b_mc_m(g,n)
\end{align*}
where for all $m\geq 1$,
$$b_m=\frac{m\cdot\pi^{2m-2}}{(2m+1)!}\text{\quad and \ }c_m(g,n)=\frac{[\tau_m\cdot\tau_0^n]_{g,n+1}}{V_{g,n+1}}.$$
It is clear that $\{b_m\}_{m\geq 0}\subset \R_{\geq 0}$ is decreasing. And it is known by \cite[Theorem 6.1]{Mirz07a} and \cite[Lemma 3.2]{Mirz13} that $\{c_m(g,n)\}_{m\geq 0}\subset \R_{\geq 0}$ is also decreasing. In particular we have
\[c_1(g,n)\leq c_0(g,n)=1.\]
Denote by
$$S_k=\sum\limits_{l=1}^{k-1}(-1)^{l-1}c_l(g,n).$$
Then we have
\begin{enumerate}
\item $S_1=0,\ S_2=c_1(g,n)$.
\item $c_1(g,n)-c_2(g,n)\leq S_k\leq c_1(g,n)\ (k\geq 2)$.
\end{enumerate}
   Apply Abel transformation, we have
\begin{align*}
&\sum\limits_{m=1}^{3g-2+n}(-1)^{m-1}b_mc_m(g,n)=\sum\limits_{m=1}^{3g-2+n}b_m(S_{m+1}-S_m)\\
&=\sum\limits_{m=1}^{3g-2+n}S_{m+1}(b_m-b_{m+1})+b_{3g-1+n}S_{3g-1+n}\\
&\leq c_1(g,n)\cdot \left(\sum\limits_{m=1}^{3g-2+n}(b_m-b_{m+1})\right)+b_{3g-1+n}S_{3g-1+n}\\
&< \frac{c_1(g,n)}{6}+b_{3g-1+n}\cdot c_1(g,n) \quad (\text{because $b_1=\frac{1}{6}$})
\end{align*}
which implies that
\begin{align}\label{e-sup}
\limsup\limits_{g+n\to\infty}\frac{(2g-2+n)V_{g,n}}{V_{g,n+1}}\leq\limsup\limits_{g+n\to\infty}\frac{c_1(g,n)}{12}\leq\frac{1}{12}.
\end{align}
We also have
\begin{align*}
&\sum\limits_{m=1}^{3g-2+n}(-1)^{m-1}b_mc_m(g,n)=\sum\limits_{m=1}^{3g-2+n}S_{m+1}(b_m-b_{m+1})+b_{3g-1+n}S_{3g-1+n}\\
&\geq c_1(g,n)(b_1-b_2)+(c_1(g,n)-c_2(g,n))\cdot \left(\sum\limits_{m=2}^{3g-2+n}(b_m-b_{m+1})\right)\\
&+b_{3g-1+n}\cdot (c_1(g,n)-c_2(g,n))=\frac{1}{6}c_1(g,n)-\frac{\pi^2}{60}c_2(g,n)\geq \left(\frac{1}{6}-\frac{\pi^2}{60} \right)c_1(g,n)
\end{align*}
which implies
\begin{align}\label{e-inf}
\liminf\limits_{g+n\to\infty}\frac{(2g-2+n)V_{g,n}}{V_{g,n+1}}\geq\liminf\limits_{g+n\to\infty}\left(\frac{1}{12}-\frac{\pi^2}{120}\right)c_1(g,n)\geq\frac{1}{24}-\frac{\pi^2}{240},
\end{align}
where in the second inequality we apply one property in \cite[Page 279]{Mirz13} saying that
$$c_1(g,n)=\frac{\left[\tau_1\tau_0^{n-1}\right]_{g,n}}{V_{g,n}}\geq\frac{1}{2}.$$
Then the LHS inequality clearly follows from \eqref{e-sup} and \eqref{e-inf}.\\

Now we prove the RHS inequality. Write $g+n=M$ for simplicity. 

\underline{Case I: $n\leq g^{\frac{3}{4}}$}. From Lemma \ref{l-MZ}, we have that for $M$ large enough,
\begin{align*}
\left|\frac{(2g-2+n)V_{g,n}}{V_{g,n+1}}-\frac{1}{4\pi^2}\right|\leq\frac{c_2}{g^{\frac{1}{4}}}\leq\frac{2c_2}{M^{\frac{1}{4}}}
\end{align*}
and
$$\left|\frac{(2g-3+n)V_{g,n-1}}{V_{g,n}}-\frac{1}{4\pi^2}\right|\leq\frac{2c_2}{M^{\frac{1}{4}}}.$$
Combine these two inequalities above, one may conclude that for $M$ large enough,
\begin{align}\label{e-sup-1}
\frac{V_{g,n}^2}{V_{g,n-1}V_{g,n+1}}\leq\frac{M^{\frac{1}{4}}+8\pi^2c_2}{M^{\frac{1}{4}}-8\pi^2c_2}.
\end{align}
By letting $M\to \infty$, this implies the RHS inequality in this first case.

\underline{Case II: $n\geq g^\frac{3}{4}$}. Set $L=\frac{g^{\frac{1}{8}}}{n^{\frac{1}{4}}}.$
In this case it is clear that \[\lim \limits_{g+n\to \infty}L=0.\]
It is not hard to see that
\begin{align}\label{e-sup-3}
n \succ M^{\frac{3}{4}},\ L\prec \left(\frac{1}{M}\right)^{\frac{1}{16}}\text{ \  and \ }\frac{M}{n^{\frac{3}{2}}g^{\frac{1}{4}}}\prec \left(\frac{1}{M}\right)^{\frac{1}{4}}.
\end{align}
For large enough $(g+n)$, the proof of Equation \eqref{exp-e1} yields that
$$\mathbb{E}_{\textnormal{WP}}^{g,n+1}\left[\left(N_{(0,3)}^{(g,n)}(X,L)\right)^2\right]=\mathbb{E}_{\textnormal{WP}}^{g,n+1}\left[N_{(0,3)}^{(g,n)}(X,L)\right]+\mathbb{E}_{\textnormal{WP}}^{g,n+1}\left[N_{(0,3),2}^{(g,n-1)}(X,L)\right].$$
Together with Lemma \ref{l-cs}, it follows that
\begin{align}\label{e-sup-4}
1&\geq\textnormal{Prob}_{\textnormal{WP}}^{g,n+1}\left(X\in\mathcal{M}_{g,n+1};\ N_{(0,3)}^{(g,n-1)}(X,L)\geq 1\right)\\
&\geq\frac{\mathbb{E}_{\textnormal{WP}}^{g,n+1}\left[N_{(0,3)}^{(g,n)}(X,L)\right]^2}{\mathbb{E}_{\textnormal{WP}}^{g,n+1}\left[\left(N_{(0,3)}^{(g,n)}(X,L)\right)^2\right]}\nonumber\\
&\geq\frac{\mathbb{E}_{\textnormal{WP}}^{g,n+1}\left[N_{(0,3)}^{(g,n)}(X,L)\right]^2}{\mathbb{E}_{\textnormal{WP}}^{g,n+1}\left[N_{(0,3)}^{(g,n)}(X,L)\right]+\mathbb{E}_{\textnormal{WP}}^{g,n+1}\left[N_{(0,3),2}^{(g,n-1)}(X,L)\right]}.\nonumber
\end{align}
From \eqref{e-sup-3} and Lemma \ref{l-exp-s}, we have
\begin{align}\label{e-sup-5}
\mathbb{E}_{\textnormal{WP}}^{g,n+1}\left[N_{(0,3)}^{(g,n)}(X,L)\right]&=n(n+1)\cdot\frac{V_{g,n}}{V_{g,n+1}}\cdot \frac{L^2}{4}(1+O(L))\\
&= \frac{V_{g,n}}{V_{g,n+1}}\cdot\frac{n^2L^2}{4}\cdot \left(1+A\right)\nonumber
\end{align}
and
\begin{align}\label{e-sup-6}
\mathbb{E}_{\textnormal{WP}}^{g,n+1}\left[N_{(0,3),2}^{(g,n-1)}(X,L)\right]
&=\prod\limits_{i=0}^3 (n+1-i)\cdot\frac{V_{g,n-1}}{V_{g,n+1}}\cdot\frac{L^4}{16}(1+O(L))\\
&=\frac{V_{g,n-1}}{V_{g,n+1}}\cdot\frac{n^4L^4}{16}\cdot \left(1+B\right),\nonumber
\end{align}
where $A,B=O\left(\frac{1}{M^{\frac{1}{16}}}\right)$ are positive and the implied constants are independent of $g$ and $n$. Then it follows from \eqref{e-sup-4}, \eqref{e-sup-5} and \eqref{e-sup-6} that 
\begin{align}\label{e-sup-8}
1 &\geq \frac{\left(\frac{V_{g,n}}{V_{g,n+1}}\cdot\frac{n^2L^2}{4}\right)^2\left(1+A\right)^2}{\left(\frac{V_{g,n-1}}{V_{g,n+1}}\cdot\frac{n^4L^4}{16}+\frac{V_{g,n}}{V_{g,n+1}}\cdot\frac{n^2L^2}{4}\right)\left(1+\max\{A,B\}\right)}\\
&\geq\frac{\frac{V_{g,n}^2}{V_{g,n-1}V_{g,n+1}}}{1+\frac{V_{g,n}}{V_{g,n-1}}\cdot\frac{4}{n^2L^2}}\cdot\frac{1}{1+\max\{A,B\}}.\nonumber
\end{align} 
Hence
\begin{align}\label{e-sup-7}
    \frac{V_{g,n}^2}{V_{g,n-1}V_{g,n+1}}\leq \left(1+\frac{V_{g,n}}{V_{g,n-1}}\cdot\frac{4}{n^2L^2}\right)\cdot\left(1+O\left(\frac{1}{M^{\frac{1}{16}}}\right)\right).
\end{align}
From \eqref{e-sup-3} and Part (2) of Lemma \ref{l-mirz}, one may conclude that
\begin{align*}
\frac{V_{g,n}}{V_{g,n-1}}\cdot\frac{4}{n^2L^2}\asymp\frac{M}{n^{\frac{3}{2}}g^{\frac{1}{4}}}\prec \frac{1}{M^{\frac{1}{4}}}.
\end{align*}
 Together with \eqref{e-sup-7}, one may conclude for $M$ large enough,
\begin{align}\label{e-sup-2}
\frac{V_{g,n}^2}{V_{g,n-1}V_{g,n+1}}\leq \left(1+O\left(\frac{1}{M^{\frac{1}{16}}}\right)\right)\cdot \left(1+O\left(\frac{1}{M^{\frac{1}{4}}}\right)\right).
\end{align}

\noindent By letting $M\to \infty$, the RHS inequality for this second case follows from \eqref{e-sup-2}.

The proof is complete.
\end{proof}
\vspace{.2in}

\subsection*{Conflict of interest} The authors have no Conflict of interest to declare that are relevant to the content of this
article.
\vspace{.1in}
\subsection*{Data Availability} The manuscript has no associated data.
\vspace{.2in}

\bibliographystyle{plain}
\bibliography{ref}
\end{document}